\documentclass[12pt]{preprint}

\usepackage[a4paper,innermargin=1.2in,outermargin=1.2in,
bottom=1.5in,marginparwidth=1in,marginparsep
=3mm]{geometry}
\usepackage[T1]{fontenc} 

\usepackage[greek, greek.ancient, english]{babel} 
\usepackage{amsmath,amsfonts,amsthm} 

\usepackage{shuffle}
\usepackage{amssymb}
\usepackage{fancyhdr} 
\usepackage{comment}
\usepackage{margincomment}
\usepackage{url}
\usepackage{cprotect}
\usepackage{hyperref}
\usepackage{mathrsfs}
\usepackage{mathtools}
\usepackage{mhequ}
\usepackage{mhsymb}
\usepackage{microtype}
\usepackage{colonequals}
\usepackage{braket}

\usepackage{xstring}

\usepackage{orcidlink}

\usepackage{fourier} 

\providecommand{\llbracket}{\left[\!\left[}
\providecommand{\rrbracket}{\right]\!\right]}

\makeatletter
\newcommand{\mbig}[1]{\vcenter{\hbox{\scalebox{1}[1.3]{$#1$}}}}
\makeatother

\renewcommand{\digamma}{\text{\foreignlanguage{greek}{\ddigamma}}}

\colorlet{darkblue}{blue!90!black}
\colorlet{darkred}{red!90!black}
\colorlet{darkgreen}{green!50!black}

\newcommand{\tw}{\mathbin{\#}}

\DeclareMathOperator{\Ent}{Ent}

\DeclareMathOperator{\Var}{Var}
\DeclareMathOperator{\logd}{log^{[1]}}

\newcommand{\RR}{\mathbb{R}}

\let\Re\relax
\DeclareMathOperator{\Re}{\mathfrak{R}}

\DeclareMathOperator{\Lin}{{\mathcal{L}}}

\def\CA{{\cal A}}

\def\CE{{\cal E}}

\def\CH{{\cal H}}
\def\CI{{\cal I}}

\def\CL{{\cal L}}
\def\CM{{\cal M}}

\def\CP{{\cal P}}

\def\CS{{\cal S}}

\def\CW{{\cal W}}

\DeclareMathOperator{\Id}{Id}

\def\op{\mathrm{op}}
\def\Aop{A^{\mathrm{op}}}
\def\mop{\mcb{m}^{\hspace{-.1cm}\mathrm{op}}}
\def\pop{\CP^{\mathrm{op}}}
\def\bop{\bone^{\mathrm{op}}}

\def\bone{\mathbf{1}}

\def\sa{\mathrm{sa}}

\def\eqdef{\coloneqq}

\def\Wick#1{\mathord{\kern0.1em{:}#1{:}\kern0.1em}}

\def\d{\mbox{d}}

\DeclareFontFamily{U}{BOONDOX-calo}{\skewchar\font=45 }
\DeclareFontShape{U}{BOONDOX-calo}{m}{n}{
  <-> s*[1.05] BOONDOX-r-calo}{}
\DeclareFontShape{U}{BOONDOX-calo}{b}{n}{
  <-> s*[1.05] BOONDOX-b-calo}{}
\DeclareMathAlphabet{\mcb}{U}{BOONDOX-calo}{m}{n}
\SetMathAlphabet{\mcb}{bold}{U}{BOONDOX-calo}{b}{n}

\makeatletter 
\newcommand*{\bigcdot}{}
\DeclareRobustCommand*{\bigcdot}{%
  \mathbin{\mathpalette\bigcdot@{}}%
}
\newcommand*{\bigcdot@scalefactor}{.7}
\newcommand*{\bigcdot@widthfactor}{1.15}
\newcommand*{\bigcdot@}[2]{%
  \sbox0{$#1\vcenter{}$}
  \sbox2{$#1\cdot\m@th$}%
  \hbox to \bigcdot@widthfactor\wd2{%
    \hfil
    \raise\ht0\hbox{%
      \scalebox{\bigcdot@scalefactor}{%
        \lower\ht0\hbox{$#1\bullet\m@th$}%
      }%
    }%
    \hfil
  }%
}
\makeatother

\newcommand{\vnotimes}{\mathbin{\mkern1mu\overline{\mkern-2.5mu\otimes\mkern-2.5mu}\mkern1mu}}

\newcommand{\vertiii}[1]{{\left\vert\kern-0.25ex\left\vert\kern-0.25ex\left\vert #1 
		\right\vert\kern-0.25ex\right\vert\kern-0.25ex\right\vert}}

\newtheorem{theorem}{Theorem}[section]

\newtheorem{corollary}[theorem]{Corollary}

\newtheorem{lemma}[theorem]{Lemma}
\newtheorem{definition}[theorem]{Definition}

\newtheorem{proposition}[theorem]{Proposition}
\newtheorem{assumption}[theorem]{Assumption}
\theoremstyle{remark}
\newtheorem{remark}[theorem]{Remark}
\newtheorem{example}[theorem]{Example}

\numberwithin{equation}{section} 
\numberwithin{figure}{section} 
\numberwithin{table}{section} 

\colorlet{symbols}{blue!90!black}
\colorlet{testcolor}{green!60!black}

\usetikzlibrary{calc}

\tikzset{
	eps/.style={circle,fill=white,draw=symbols,inner sep=0pt,minimum size=0.8mm},
	}
\makeatletter
\def\DeclareSymbol#1#2#3{\expandafter\gdef\csname MH@symb@#1\endcsname{\tikz[baseline=#2,scale=0.15,draw=symbols]{#3}}}
\def\<#1>{\csname MH@symb@#1\endcsname}
\makeatother

\makeatletter

\DeclareRobustCommand{\TitleEquation}[2]{\texorpdfstring{\StrLeft{\f@series}{1}[\@firstchar]$\if%
		b\@firstchar\boldsymbol{#1}\else#1\fi$}{#2}}

\makeatother

\def\Poly{\mathcal{P}}
\def\Words{\mathcal{W}}
\def\indet{\mathcal{X}}

\newcommand{\horrule}[1]{\rule{\linewidth}{#1}} 

\title{	
\horrule{0.5pt} \\[0.4cm] 
\large Free Bakry--\'{E}mery Calculus
 \\ 
\horrule{0.5pt} \\[0.5cm] 
}
\author{Ajay Chandra\inst1\orcidlink{0000-0003-3690-1890}, Maria Gordina\inst2\orcidlink{0000-0002-9003-7362}, and Martin Peev\inst3\orcidlink{0000-0001-6191-8067}}
\institute{Purdue University, West Lafayette, USA \email{ajay.chandra@gmail.com} \and
University of Rochester, USA,  \email{mgordina@ur.rochester.edu} \and 
University of Oxford, UK, \email{martin.peev@maths.ox.ac.uk} }

\begin{document}

\maketitle 
\begin{abstract}
We prove a free probabilistic Bakry--\'{E}mery criterion for logarithmic Sobolev inequalities and hypercontractivity. Our approach relies on a noncommutative version of the carr\'{e} du champ operator and its iteration, which we are able to define in the free probability setting. This allows us to formulate a positive curvature condition, which is shown to be sufficient for logarithmic Sobolev inequalities and hypercontractivity in this setting. 
In the setting of free Gibbs measures, we also show that this curvature condition can be expressed in terms of the positivity of the Hessian of the associated potential. 
\end{abstract}



\tableofcontents

\section{Introduction}
Logarithmic Sobolev inequalities connect entropy production, convergence to equilibrium, and the regularising properties of Markov semigroups. Gross's logarithmic Sobolev inequality for Gaussian measure \cite{Gross1975c, Gross1975a, Gross2006onNelson}, together with Nelson's work on the Ornstein--Uhlenbeck semigroup \cite{Nelson1966a, Nelson1973c}, established the classical relation between logarithmic Sobolev inequalities and hypercontractivity. Bakry and \'{E}mery subsequently developed a geometric approach based on the carr\'{e} du champ $\Gamma$ and its iteration $\Gamma_2$; see, for example, \cite{Ledoux2000a, BakryGentilLedouxBook}. For a classical Langevin semigroup, a lower bound on the Hessian of the potential gives a lower curvature bound, which in turn yields gradient estimates and functional inequalities.

The purpose of this article is to develop an analogue of this Bakry--\'{E}mery calculus for free Langevin dynamics. We consider a free stochastic differential equation driven by an $m$-dimensional semicircular Brownian motion,
\begin{equ}
    \d Y_t=-DV(Y_t)\,\d t+\sqrt{2}\,\d X_t\;,
\end{equ}
where $V$ is a self-adjoint noncommutative polynomial. Our principal functional-inequality result is a density-form modified logarithmic Sobolev inequality for a stationary law $\omega$ of this dynamics. This formulation is adapted to entropy dissipation along a noncommutative Markov semigroup and should be distinguished from Gross's logarithmic Sobolev inequality for elements of the form $a^*a$.

Several forms of logarithmic Sobolev inequality and hypercontractivity were already known in noncommutative probability. Gross first studied hypercontractivity for fermionic fields \cite{Gross1975a}, and Carlen and Lieb established corresponding results for fermionic oscillator semigroups \cite{CarlenLieb1993a, CarlenLieb1993b}. Biane adapted these methods to prove logarithmic Sobolev inequalities and hypercontractivity for free Gaussian systems \cite{Biane1997c}. A general framework for hypercontractivity on noncommutative $L^p$-spaces was developed by Olkiewicz and Zegarlinski \cite{OlkiewiczZegarlinski1999}. Logarithmic Sobolev inequalities have also entered free transport theory through the work of Biane and Voiculescu \cite{BianeVoiculescu2001} and through random-matrix approximation arguments such as \cite{HiaiPetzUeda2004}. In finite-dimensional quantum probability, optimal-transport methods have yielded entropy-decay results for quantum Markov semigroups \cite{CarlenMaas2017, CarlenMaas2020, CarlenMaas2020Correction}. Noncommutative versions of Bakry--\'{E}mery theory have more recently been developed in \cite{BrannanGaoJunge2022, GaoRouze2022, GaoJungeLaRacuenteLi2025, GaoGordina2024}. We also mention the recent work \cite{SunXuZhang2026} on hypercontractivity for quantum Ornstein--Uhlenbeck semigroups through Meixner polynomials. Our aim differs from these Gaussian, finite-dimensional, or orthogonal-polynomial approaches: we formulate a curvature criterion directly for free Langevin dynamics associated with non-quadratic polynomial potentials.

The first difficulty is that the free difference quotient is tensor-valued. After passing to the opposite algebra in the first tensor factor, the gradient takes values in the enveloping algebra $\CM^e=\CM^{\op}\vnotimes\CM$. Consequently, there are two natural carr\'{e} du champ operators. The lifted quantity
\begin{equ}
    \tilde\Gamma(a)=(\nabla a)^*\cdot\nabla a
\end{equ}
is $\CM^e$-valued, while its slice
\begin{equ}
    \Gamma(a)=(\omega^{\op}\vnotimes\Id)\bigl(\tilde\Gamma(a)\bigr)
\end{equ}
is $\CM$-valued. These lead to distinct strong and weak curvature conditions and gradient estimates. This distinction is essential: when $K>0$, the weak estimate is sufficient for Poincar\'{e}-type consequences, whereas the proof of the modified logarithmic Sobolev inequality uses the completed order-valued strong estimate.

We define corresponding iterated carr\'{e} du champ operators $\Gamma_2$ and $\tilde\Gamma_2$ and compute them through commutators of the generator with the free gradient. The resulting commutator is governed by an enveloping-algebra-valued Hessian $\vec T_V$ of the potential. We introduce full and partial $K$-convexity according to whether the Hermitian part of this Hessian is bounded below before or after applying the slice map. Full $K$-convexity implies
\begin{equ}
    \tilde\Gamma_2(a)\geqslant K\tilde\Gamma(a)\;,
\end{equ}
whereas partial $K$-convexity implies the corresponding sliced inequality
\begin{equ}
    \Gamma_2(a)\geqslant K\Gamma(a)\;.
\end{equ}
At the algebraic level, the Hessian appearing here is the same tensor-valued Hessian that occurs in the notion of $h$-convexity from \cite{DabrowskiGuionnetShlyakhtenko2022}. Related commutator expressions appear in \cite{AIHPB_2014__50_4_1404_0}; we also learned during the preparation of this article that Diez independently uses such commutators to study free Poincar\'{e} inequalities and a free analogue of Obata rigidity \cite{diez2026obata}.

It is important to separate these algebraic curvature computations from the analytical construction of the semigroup. We work under hypotheses ensuring that the free stochastic differential equation gives a stationary, ultraweakly continuous Markov semigroup on a tracial von Neumann algebra. From Section~\ref{sec:free_ito} onwards, we impose the stated symmetry and polynomial-core assumptions needed to close the free gradient and justify the generator computations. These properties are standing analytical hypotheses: they are not consequences merely of the fact that $V$ is polynomial or of positivity of its algebraic Hessian. In particular, our functional-inequality argument does not use the stronger global convexity assumptions employed in constructions of free Langevin dynamics such as \cite{GuionnetShlyakhtenko2009, DabrowskiGuionnetShlyakhtenko2022}.

Under these standing hypotheses, full $K$-convexity with $K>0$ yields the strong gradient estimate
\begin{equ}
    \tilde\Gamma(P_ta)\leqslant e^{-2Kt}P_t^e\bigl(\tilde\Gamma(a)\bigr)\;.
\end{equ}
We transfer this estimate to the logarithmic-mean metric by combining complete positivity with the Choi--Davis--Jensen inequality. This gives exponential decay of the Fisher information. Integrating the entropy-dissipation identity then proves, for every density $\rho\in\CL^1(\omega)_+$,
\begin{equ}
    \Ent_\omega(\rho)\leqslant\frac{1}{2K}\CI_\omega(\rho)\;.
\end{equ}
Thus the modified logarithmic Sobolev constant in our convention is $2K$. Here a density satisfies $\omega(\rho)=1$, and the Fisher information is understood in the relaxed sense defined below, so the inequality extends from bounded uniformly positive densities to arbitrary densities. The argument does not use finite-dimensional matricial approximations and is not restricted to the free Gaussian potential.

The inequality proved here is a density-form modified logarithmic Sobolev inequality, whose dissipation is $\CE(\rho,\log(\rho))$. It is not, without an additional comparison of dissipation functionals, Gross's logarithmic Sobolev inequality for $a^*a$. Consequently, we do not claim that the result by itself implies hypercontractivity. Establishing such a comparison, or proving a Gross inequality directly under the curvature condition introduced here, is a separate problem.

The article is organised as follows. Section~\ref{sec:notation} fixes the tensor-product, opposite-algebra, and free-difference-quotient conventions. Section~\ref{sec:commutative} recalls the commutative $\Gamma$-calculus motivating the free construction. Section~\ref{sec:free_ito} introduces the free Langevin dynamics, computes its generator, and records the analytical assumptions used subsequently. Section~\ref{sec:nc_gamma} develops the closed free Dirichlet form, the two carr\'{e} du champ operators, and the Hessian commutator identities. The final section proves the gradient and Poincar\'{e} estimates, the logarithmic-mean estimate, and the density-form modified logarithmic Sobolev inequality.

\subsection*{Acknowledgements} 
MG gratefully acknowledges partial support by NSF Grant DMS-2246549.
MP gratefully acknowledges support by the UKRI Frontier Research grant \emph{Stochastic analysis of quantum fields} [grant number EP/Z534328/1] awarded to Massimiliano Gubinelli.
Part of this research was performed while AC and MG were visiting  the SLMath (\emph{Recent Trends in Stochastic Partial Differential Equations}) and Hausdorff Research Institute for Mathematics (\emph{Probabilistic methods in quantum field theory}) whose support is greatly appreciated. 

\section{Notation}\label{sec:notation}

\subsection{Tensor Products}\label{ss:tensor_products}

Given two real or complex vector spaces  $V$ and $V'$, we denote by $V \otimes V'$ the \emph{algebraic} tensor product.
We define $\iota \colon V \otimes V \rightarrow V \otimes V$ to be the involution that permutes factors, that is, $\iota(v \otimes w) = w \otimes v$.

Given an algebra $\CA$, we write $\mcb{m} \colon \CA \otimes \CA \to \CA$ for the corresponding multiplication map, although we will often suppress the multiplication map from the notation, writing $\mcb{m}(a,b) = \mcb{m}(a \otimes b) = ab$.

We also write $\amalg$ for the various natural concatenation products induced by $\mcb{m}$. That is, given simple tensors $ v \otimes a \in (V \otimes \CA)$ and $b \otimes w \in \mathcal{A} \otimes W$, we set 
\begin{equ}
    (v \otimes a) \amalg (b \otimes w) = 
    v \otimes ab \otimes w \in V \otimes \mathcal{A} \otimes W\;.
\end{equ}
In particular, when $V$ or $W$ is just a field of scalars, this is a standard $\mathcal{A}$-module action and we suppress $\amalg$ from the notation.

For any algebra $\mathcal{A}$, we denote by $\tw$ the natural  action of $\mathcal{A} \otimes \mathcal{A}$ on $\mathcal{A}$ given by $(a \otimes b) \tw c = acb$.
We overload this notation, and also use $\tw$ to denote the action of $\mathcal{A} \otimes \mathcal{A}$ on $\mathcal{A} \otimes \mathcal{A}$ given by setting $(a \otimes b) \tw (c \otimes d) = ac \otimes db$.

\subsection{Positivity} 

Within this article we will use the following equivalent characterisation of positivity for elements $a \in \CA$ of a $C^*$-algebra $\CA$: for $a \in \CA_{\sa}$
\begin{equ}[eq:positivity]
    a \geqslant 0  \quad \iff \quad \forall \tau \in \CS(\CA) : \tau(a) \geqslant 0 \; ,
\end{equ} 
where $\CS(\CA)$ is the set of states of $\CA$. 

For a finite trace $\omega$ on $\CA$ and $p \in [1,\infty]$, let $\CL^p(\omega)$ denote the usual noncommutative $L^p$-spaces. An element $a \in \CL^1(\omega)$ is non-negative if and only if, for all $b \in \CA_+$,
\begin{equ}
    \omega(ab) \geqslant 0 \; . 
\end{equ}
We denote the set of such non-negative elements by $\CL^1(\omega)_+$.

\subsection{Polynomial Observable Algebra}\label{ss:poly_algebra}
Most of our computations will occur on a $*$-algebra of polynomial observables, which we denote by $\Poly$.
In particular, we fix $m \geqslant 1$ and set $\Poly \eqdef \C \langle \indet^{1},\dots, \indet^{m} \rangle$ where $\indet^{1},\dots,\indet^{m}$ are noncommutative self-adjoint indeterminates.
A word $w = (i_{1},\dots,i_n)$ is a finite (possibly empty) ordered list of elements of $[m]\eqdef \{1,\dots,m\}$.
Given two words $w, w'$, we write $w w'$ for their concatenation.
We also write $|w|$ for the length of the word $w$. We write $\Words$ for the collection of all words, $\emptyset$ for the empty word, and, for $i \in [m]$, $i$ is the word that consists of only one instance of $i$.

Given a word $w$, we write $\indet^{w} = \prod_{i \in w} \indet^{i}$ where the product is ordered and allows for multiplicity.
Clearly, the collection of monomials $\left\{ \indet^{w} \,\mbig|\,
w \in \CW\right\}$ is a basis for $\Poly$.

The $\ast$-algebra structure on $\Poly$ is the one for which the indeterminates are self-adjoint: $\ast$ is the unique antilinear anti-automorphism of $\Poly$ fixing each $\indet^{i}$. Writing $\overleftarrow{w} = (i_{n},\dots,i_{1})$ for the reversal of a word $w = (i_{1},\dots,i_{n})$, $*$ acts on monomials by $(\indet^{w})^{\ast} \eqdef \indet^{\overleftarrow{w}}$.

For $i \in [m]$, we define the corresponding (noncommutative) partial derivative as a linear map $\partial_{i} \colon \Poly \to \Poly \otimes \Poly$ given by setting, for any $\indet^{w}$,
\begin{equ}
    \partial_{i} \indet^{w}
    \eqdef
    \sum_{w = w_1 i w_2} \indet^{w_1} \otimes \indet^{w_2} \;.
\end{equ}
We note that, for $a,b \in \Poly$, we have the Leibniz rule
\begin{equ}[eq:Leibniz]
    \partial_{i} (ab) =
    (\partial_{i}a)b + a(\partial_{i}b)\;,
\end{equ}
where on the right we use that $\Poly \otimes \Poly$ is a $\Poly$ bimodule. 
We also write $\partial_{i}\partial_{j} = (\partial_{i} \otimes \Id)\partial_{j}$.
We denote the corresponding cyclic partial derivative by $D_{i} \colon \Poly \to \Poly$, defined as $D_{i} \eqdef \mcb{m} \circ \iota \circ \partial_{i}$, that is, 
\begin{equ}
    D_{i} \indet^{w}
    =
    \sum_{w = w_1 i w_2} \indet^{w_2}\indet^{w_1} \;.
\end{equ}
We denote by $\partial \colon \Poly \rightarrow (\Poly \otimes \Poly)^m$ and $D \colon \Poly \rightarrow \Poly^{m}$ the corresponding ``gradient'' maps, that is $\partial P = (\partial_{i}P)_{i =1}^{m}$ and similarly for $D$.

We introduce a noncommutative Laplacian $\Delta \colon \Poly \rightarrow \Poly^{\otimes 3}$ by setting
\begin{equ}
    \Delta \eqdef
    \sum_{i=1}^{m} \big( \partial_{i} \otimes \Id + \Id \otimes \partial_{i}   \big)
    \partial_{i}
    =
    2 \sum_{i=1}^{m}
    (\partial_{i} \otimes \Id)\partial_{i}\;.
\end{equ}
We equip $\Poly^{\otimes 2}$ with the component-wise product~$\cdot$ and involution $(a \otimes b)^{\ast} = a^{\ast} \otimes b^{\ast}$.

We adopt the notational convention that $\partial_{i}a^{\ast}$ means one takes the adjoint before taking the partial free derivative. 
In particular, $\ast$-operation interacts with differentiation via
\begin{equ}\label{eq:partial_star}
   \partial_{i}a^{\ast} = \partial_{i}(a^{\ast}) = (\iota\,\partial_{i}a)^{\ast}\;;
\end{equ}
equivalently, $\partial_{i}(a^{\ast}) = \sum (a^{(2)})^{\ast} \otimes (a^{(1)})^{\ast}$ when $\partial_{i}a = \sum a^{(1)} \otimes a^{(2)}$.

\section{Commutative Gamma Calculus and Functional Inequalities}\label{sec:commutative}

In this section, we review the classical Bakry--\'{E}mery $\Gamma$-calculus and its applications to functional inequalities.
The material presented here is well-established and serves as a motivation for the noncommutative extensions developed in subsequent sections. Moreover, we restrict the exposition to Markov semigroups over Euclidean spaces, as we are mostly concerned with differential operators on such spaces. 

\subsection{Carr\'{e} du Champ Operator \TitleEquation{\Gamma}{Gamma} and Iterated Carr\'{e} du Champ Operator \TitleEquation{\Gamma_2}{Gamma2}}
\label{ss:cdc}

We consider a differential operator of the form
\begin{equ}
    L = \sum_{i=1}^{n} X_{i}^{2} + X_{0}\;,
\end{equ}
where $X_{0}, X_{1},\dots, X_{n}$ are smooth vector fields on $\R^{n}$.
Recall that for any vector field $X$ and smooth functions $f, g$ on $\R^{n}$, we have the product rules
\begin{equs}
    X(fg) &= (Xf) \cdot g + f \cdot (Xg)\;,\\
    X^{2}(fg) &= (X^{2}f) \cdot g + f \cdot (X^{2}g) + 2(Xf) \cdot (Xg)\;.
\end{equs}
It follows that
\begin{equ}[eq:comm_diffusion_product]
L(fg) = (Lf) \cdot g + f \cdot (Lg) + 2\sum_{i=1}^{n} (X_{i}f)(X_{i}g)\;.
\end{equ}

\begin{definition}[Carr\'{e} du Champ]\label{def:comm_gamma}
    The \emph{carr\'{e} du champ} operator $\Gamma$ is the symmetric bilinear form on smooth functions defined by
    \begin{equ}
    \Gamma(f,g) \eqdef \frac{1}{2}\Big( L(fg) - f \cdot Lg - (Lf) \cdot g \Big) = \sum_{i=1}^{n} (X_{i}f)(X_{i}g)\;.
    \end{equ}
    As is standard, we write $\Gamma(f) \eqdef \Gamma(f,f)$.
    \end{definition}
    
    With this notation, \eqref{eq:comm_diffusion_product} takes the form
    \begin{equ}[eq:comm_product_compact]
        L(fg) = (Lf) \cdot g + f \cdot (Lg) + 2\,\Gamma(f,g)\;.
    \end{equ}
    
    \begin{definition}[Iterated carr\'{e} du champ]\label{def:comm_gamma2}
    The \emph{iterated carr\'{e} du champ} is the symmetric bilinear form
    \begin{equ}
        \Gamma_{2}(f,g) \eqdef \frac{1}{2}\Big( L\,\Gamma(f,g) - \Gamma(f, Lg) - \Gamma(Lf, g) \Big)\;.
    \end{equ}
\end{definition}

By a direct computation using \eqref{eq:comm_product_compact}, one obtains the following
\begin{equ}[eq:comm_gamma2_explicit]
    \Gamma_{2}(f,g) = \frac{1}{2}\sum_{i=1}^{n}\Big( [L,X_{i}]f\cdot X_{i}g+ X_{i}f \cdot [L,X_{i}]g \Big) + \sum_{i,j=1}^{n} (X_{i}X_{j}f)(X_{i}X_{j}g)\;.
\end{equ}

We now record the standard notation used throughout. 
\begin{equs}
\nabla f &\eqdef (X_{1}f, \dots, X_{n}f)\;,\\
\Gamma(f) &= \Gamma(f,f) = \sum_{i=1}^{n} (X_{i}f)^{2} = |\nabla f|^{2}\;,\\
\operatorname{Hess} f &\eqdef \{X_{i}X_{j}f\}_{i,j=1}^{n}\;,\\
\|\operatorname{Hess} f\|^{2} &\eqdef \sum_{i,j=1}^{n} (X_{i}X_{j}f)^{2}\;,\\
\Gamma_{2}(f) &\eqdef \Gamma_{2}(f,f) = \sum_{i=1}^{n} [L,X_{i}]f \cdot X_{i}f + \|\operatorname{Hess} f\|^{2}\;.
\end{equs}
Note that, so far, we have not introduced the function spaces we will work with. Let $\left( \R^{n}, \mfB\left( \R^{n} \right) \right)$ be the Borel $\sigma$-algebra, and let $\mu$ be a (Borel) probability measure on $\left( \mathbb{R}^{n}, \mfB\left( \R^{n} \right) \right)$; typically $\mu$ is an invariant measure for the corresponding semigroup, as discussed later in Section~\ref{ss:invariant_measures}. Throughout this section we work under the following standing assumption.

\begin{assumption}[Standard Algebra]\label{ass:standard_algebra}
    There is a subalgebra $\CA \subset L^{2}\left( \d\mu \right)$, containing the constants and dense in $L^{2}\left( \d\mu \right)$, which is a core for $L$ and is stable under $L$ and under the semigroup $(P_{t})_{t \geqslant 0}$. All $\Gamma$-calculus computations below are performed on $\CA$; the resulting functional inequalities are then extended to their natural domains by density.
\end{assumption}

For the construction of such an algebra, and for the density arguments underlying the extension, we refer to \cite[Sections~1.13 and 3.1--3.3]{BakryGentilLedouxBook}. We note that stability under the semigroup is the delicate requirement: $C_{c}^{\infty}\left( \R^{n} \right)$ is dense and stable under $L$, but not under $P_{t}$; this is resolved in \textit{loc.\ cit.\ }via essential self-adjointness. These requirements are satisfied for the Langevin dynamics considered below. We denote by $\mathcal{D}_{L}$ the domain of $L$.

In the noncommutative setting, the polynomial algebra $\Poly$ plays the role of $\CA$ as a core stable under the generator, while statements involving the semigroup pass to the closure of the gradient; see Section~\ref{sec:nc_gamma}.

\begin{definition}[Curvature-Dimension Condition]\label{def:CDcondition}
    We say that the \emph{curvature-dimension inequality} $\operatorname{CD}(\kappa, n)$ holds for some $\kappa \in \mathbb{R}$ and $n \in [1,\infty]$ if for all smooth $f$
    \begin{equ}[eq:CD]
        \Gamma_{2}(f) \geqslant \kappa\, \Gamma(f) + \frac{1}{n}(Lf)^{2}\;.
    \end{equ}
\end{definition}

\begin{example}[Langevin Dynamics]\label{ex:langevin}
    Consider $L = \Delta - \nabla V \cdot \nabla$ on $\mathbb{R}^{n}$.
    Then $X_{j} = \partial_{j}$ for $j = 1,\dots, n$ and $X_{0} = -\sum_{j=1}^{n} (\partial_{j}V)\partial_{j}$.
    A straightforward computation gives
    \begin{equ}
        [L,X_{i}] = [X_{0},X_{i}] = \sum_{j=1}^{n} (\partial_{i} \partial_{j} V) \cdot \partial_{j}\;, \quad i = 1,\dots,n\;,
    \end{equ}
    and therefore for a smooth function on $\R^{n}$
    \begin{equs}
        \Gamma(f) &= |\nabla f|^{2}\;,\\
        \Gamma_{2}(f) &= \|\operatorname{Hess} f\|^{2} + \operatorname{Hess}(V)(\nabla f, \nabla f)\;.
    \end{equs}
    Note that $\operatorname{Hess}(V) \geqslant 0$ whenever $V$ is convex (equivalently, $e^{-V}$ is log-concave), and in this case $\Gamma_{2}(f) \geqslant 0$.
    More precisely, Bakry--\'{E}mery's criterion states that
    \begin{equ}
        \Gamma_{2} \geqslant \kappa\,\Gamma \quad \text{if and only if} \quad \operatorname{Hess}(V) \geqslant \kappa I_{n}\;,
    \end{equ}
    where the latter inequality is in the sense of matrices.
\end{example}

\subsection{The \TitleEquation{\operatorname{CD}(0,\infty)}{CD(0,inf)} Condition and Functional Inequalities}
\label{ss:CD0}
We now assume that $\left( L, \mathcal{D}_{L}\right)$ is the infinitesimal generator for a Markov semigroup $\left\{P_{t}\right\}_{t \geqslant 0}$ on $L^{2}\left( d\mu \right)$. 

\begin{proposition}[Convexity]\label{prop:CD0_convexity}
If $\operatorname{CD}(0,\infty)$ holds, that is, $\Gamma_{2}(f) \geqslant 0$ for all  $f \in \mathcal{A}$, then the functional $\Phi$ defined by
\begin{equ}[eq:PhiGamma]
    \Phi(s) \eqdef \Phi_{\Gamma}(s) \eqdef P_{s}\big((P_{t-s}f)^{2}\big)\;, \quad 0 \leqslant s \leqslant t, \; f \in \CA\;,
\end{equ}
 is convex on $[0,t]$.
\end{proposition}
\begin{proof}
Let $F = F(s) \eqdef P_{t-s}f$, we have $F'(s) = -LF$, and therefore by \eqref{eq:comm_product_compact},
\begin{equ}[eq:Phi_first_deriv]
    \Phi'(s) = P_{s}\big(L(F^{2}) + 2F \cdot F'(s)\big) = 2P_{s}\big(\Gamma(F)\big)\;.
\end{equ}
Since $(\Gamma(F(s)))' = -2\sum_{i=1}^{n}(X_{i}F)(X_{i}LF)$, applying \eqref{eq:comm_product_compact} to $L((X_{i}F)^{2})$ yields
\begin{equ}[eq:Phi_second_deriv]
    \Phi''(s) = 4P_{s}\,\Gamma_{2}(P_{t-s}f) \geqslant 0
\end{equ}
by the assumption that $\Gamma_{2} \geqslant 0$.
\end{proof}

\begin{proposition}[Gradient estimate]
\label{prop:comm_gradient_estimate}
The following are equivalent
\begin{enumerate}
\item $\Gamma_{2} \geqslant \kappa\,\Gamma$\;;
\item For every  $f \in \mathcal{A}$ and $t \geqslant 0$,
\begin{equ}[eq:comm_gradient_estimate]
\Gamma(P_{t}f) \leqslant e^{-2\kappa t}\, P_{t}\Gamma(f)\;.
\end{equ}
\end{enumerate}
\end{proposition}

\begin{proof}
For $0 \leqslant s \leqslant t$, define
\begin{equ}[eq:Psi_comm]
\Psi(s) \eqdef e^{-2\kappa s}\, P_{s}\big(\Gamma(P_{t-s}f)\big)\;.
\end{equ}
Then
\begin{equ}
\Psi'(s) = e^{-2\kappa s}\, P_{s}\Big( -2\kappa\,\Gamma(P_{t-s}f) + 2\,\Gamma_{2}(P_{t-s}f) \Big) \geqslant 0
\end{equ}
provided $\operatorname{CD}(\kappa,\infty)$ holds applied to $P_{t-s}f$.
Thus $\Psi(0) \leqslant \Psi(t)$, which gives \eqref{eq:comm_gradient_estimate}.

Conversely, taking $t \to 0$ in the inequality $e^{-2\kappa t}P_{t}\Gamma(f) - \Gamma(P_{t}f) \geqslant 0$ and dividing by $2t$ yields
\begin{equ}
\Gamma_{2}(f) - \kappa\,\Gamma(f) \geqslant 0\;.
\end{equ}
\end{proof}

The first consequence of the curvature-dimension inequality is the local Poincar\'{e} inequality, for a more general result we refer to \cite[Theorem 4.7.2]{BakryGentilLedouxBook}.  
\begin{corollary}[Local Poincar\'{e} inequality from $\operatorname{CD}(\kappa,\infty)$]\label{cor:PI_from_CD}
If $\Gamma_{2} \geqslant \kappa\,\Gamma$, then
\begin{equ}
\operatorname{Var}_{P_{t}}(f) \leqslant \frac{1 - e^{-2\kappa t}}{\kappa}\, P_{t}\Gamma(f)\;,
\end{equ}
where $\frac{1 - e^{-2\kappa t}}{\kappa}$ is understood to be $2t$ for $\kappa=0$. 
\end{corollary}
\begin{proof}
Let $\Phi(s) \eqdef P_{s}((P_{t-s}f)^{2})$.
Then $\Phi^{\prime}(s) = 2P_{s}(\Gamma(P_{t-s}f)) \leqslant 2e^{-2\kappa(t-s)}\,P_{t}\Gamma(f)$ for all $s \in [0,t]$ by Proposition~\ref{prop:comm_gradient_estimate}. Therefore
\begin{equs}
& P_{t}\left( f^{2}\right)-\left( P_{t} f \right)^{2}=\Phi\left( t\right)-\Phi\left( 0\right)=\int_{0}^{t} \Phi^{\prime}\left( s\right)ds
\\
&\leqslant 2\int_{0}^{t} e^{-2\kappa(t-s)}P_{t}\Gamma(f)\, ds=2P_{t}\Gamma(f)\int_{0}^{t} e^{-2\kappa(t-s)}\, ds=\frac{1-e^{-2\kappa t}}{\kappa}P_{t}\Gamma(f)
\end{equs}
\end{proof}

\subsection{Invariant Measures, Ergodicity and Functional Inequalities}
\label{ss:invariant_measures}

Consider a Markov semigroup $P_{t}$ on $L^{2}\left( d\mu \right)=L^{2}\left( \mathbb{R}^{n}, d\mu \right)$. We will now review assumptions on the measure $\mu$ that allow us to prove functional inequalities for the infinitesimal generator $L$. We say that $\mu$ is an \emph{invariant measure} for the semigroup $P_{t}$ if for any $f \in L^{1}(\d \mu)$
\begin{equ}
    \int_{\mathbb{R}^{n}} P_{t}f\,\d \mu = \int_{\mathbb{R}^{n}} f\,\d \mu\;,
\end{equ}
and $\mu$ is \emph{time-reversible} for $P_{t}$ (equivalently, $P_{t}$ is \emph{symmetric} with respect to $\mu$) if for any $f, g \in L^{2}(\d \mu)$
\begin{equ}
\int_{\mathbb{R}^{n}} f\, P_{t}g\,\d \mu = \int_{\mathbb{R}^{n}} (P_{t}f)\, g\,\d \mu\;.
\end{equ}
Since $\mu$ is a finite measure, time-reversible measures are invariant whenever $P_{t}$ is conservative (take $g \equiv 1$).
Infinitesimally, these conditions are characterised by
\begin{equ}
\int_{\mathbb{R}^{n}} f\, Lg\,\d \mu = \int_{\mathbb{R}^{n}} (Lf)\, g\,\d \mu
\text{ and } 
\int_{\mathbb{R}^{n}} Lf\,\d \mu = 0
\end{equ}
for all $f, g$ in the domain of $L$.

The semigroup $P_{t}$ is \emph{ergodic} with respect to $\mu$ if for any  $f \in L^{2}\left( d\mu \right)$
\begin{equ}
    P_{t}f \xrightarrow[t \to \infty]{} \int_{\mathbb{R}^{n}} f\,\d \mu\;.
\end{equ}
By  \cite[Section~3.1.9]{BakryGentilLedouxBook} we see that for a finite measure $\mu$ such that $\Gamma\left( f \right)=0$ only for constant functions, the corresponding (conservative) semigroup is ergodic with respect to $\mu$. This is the case for the Langevin semigroup.

We define the \emph{Dirichlet form} associated to $L$ and $\mu$ by the integration by parts formula
\begin{equ}[eq:IbyP]
    \CE(f,g) \eqdef \int_{\mathbb{R}^{n}}\Gamma(f,g)\,\d \mu = -\int_{\mathbb{R}^{n}} f\, Lg\,\d \mu\;.
\end{equ}
By Assumption~\ref{ass:standard_algebra}, the form $\CE$ is defined on $\CA$ and closable; we denote the domain of its closure by $\mathcal{D}_{\CE}$. Functional inequalities established for $f \in \CA$, such as the Poincar\'{e} inequality below, extend to $\mathcal{D}_{\CE}$ by density, as both sides are continuous with respect to the form norm.

Before proceeding to the goal of our exposition, logarithmic Sobolev inequalities, we start with the Poincar\'{e} inequality for $\Gamma$ and $\mu$. 
\begin{definition}[Poincar\'{e} inequality]
We say that the carr\'{e} du champ operator $\Gamma$ on $L^{2}\left( d\mu \right)$ satisfies a \emph{Poincar\'{e} inequality} with constant $C > 0$ if for all $f \in \mathcal{D}_{\CE}$, the domain of the Dirichlet form,
\begin{equ}
\operatorname{Var}_{\mu}(f) \eqdef \int_{\mathbb{R}^{n}} f^{2}\,\d \mu - \Big(\int_{\mathbb{R}^{n}} f\,\d \mu\Big)^{2} \leqslant C_{\operatorname{PI}} \,\CE(f) \eqdef C_{\operatorname{PI}} \int_{\mathbb{R}^{n}}\Gamma(f)\,\d \mu\;.
\end{equ}
\end{definition}

\begin{proposition}\label{prop:CD_implies_PI}
Suppose $P_{t}$ is a semigroup on $L^{2}\left( d\mu \right)$, where $\mu$ is an ergodic measure for $P_{t}$. If the curvature-dimension inequality~$\operatorname{CD}(\kappa,\infty)$ holds with $\kappa > 0$, then the corresponding carr\'{e} du champ operator $\Gamma$ satisfies the Poincar\'{e} inequality with constant $1/\kappa$.
\end{proposition}

\begin{proof} By Corollary~\ref{cor:PI_from_CD} with $\kappa>0$ 
and ergodicity,  and by taking $t \to \infty$ we get the result.
\end{proof}

\begin{remark}[Spectral gap and integrated $\Gamma_{2}$ criterion]\label{cor:spectral_gap}
One of the challenges in adapting $\Gamma$ calculus to noncommutative settings is that it is not easy to compare non-scalar objects. This becomes more straightforward if one replaces integrated curvature-dimension inequalities by those using a trace. In this paper we do not use this approach, but one might consider an analogue of the result below.

Let $\kappa > 0$ and suppose $\mu$ is reversible for $P_t$. Then the following are equivalent.
\begin{equ}
    \int_{E}\Gamma_{2}(f)\,\d \mu \geqslant \kappa\int_{E}\Gamma(f)\,\d \mu
\quad \Longleftrightarrow \quad
C_{\operatorname{PI}} \leqslant \frac{1}{\kappa}
\quad \Longleftrightarrow \quad
\text{spectral gap of } (-L) \geqslant \kappa\;.
\end{equ}
We refer to \cite[Section~4.2.1]{BakryGentilLedouxBook} for more details on the relation between Poincar\'{e} inequalities and spectral gap inequalities. 
\end{remark}

\subsection{Logarithmic Sobolev Inequalities}
\label{ss:LSI}

We now consider logarithmic Sobolev inequalities with respect to an invariant measure.
Define, for a convex function $\Phi$,
\begin{equ}
\operatorname{Ent}_{\mu}^{\Phi}(f) \eqdef \int_{E}\Phi(f)\,\d \mu - \Phi\Big(\int_{E} f\,\d \mu\Big) = \mathbb{E}_{\mu}\Phi(f) - \Phi(\mathbb{E}_{\mu}f)\;.
\end{equ}
The two principal choices are $\Phi(u) = u^{2}$ (which yields the variance) and $\Phi(u) = u\log u$ for $u > 0$ (which yields the entropy).

Note that by the diffusion property, $\Gamma(f^{2}) = 4f^{2}\Gamma(f)$, so
\begin{equ}
\int_{E}\frac{\Gamma(f^{2})}{f^{2}}\,\d \mu = 4\int_{E}\Gamma(f)\,\d \mu\;.
\end{equ}

For the remainder of this subsection we specialize to $\Phi(u) = u\log u$, the case relevant for the logarithmic Sobolev inequality.
We also specialize to Langevin dynamics for Gibbs measures.
Suppose $e^{-V}\d x$ is a finite measure on $\mathbb{R}^{n}$, and write
\begin{equ}
\d \mu_{V} \eqdef \frac{e^{-V}\d x}{\int e^{-V}\d x}\;.
\end{equ}
Then $L = \Delta - \nabla V \cdot \nabla$ is symmetric on $L^{2}(\mathbb{R}^{n}, \d \mu_{V})$, and $\mu_{V}$ is invariant for $L$.

We now prove the well-known fact that strict convexity of $V$ implies a Logarithmic Sobolev inequality for $\mu_V$, proceeding using Gamma calculus with respect to the Langevin type operator $L$. 
Again, we present this argument of a classical fact since we will generalize this argument to the free probability setting later in this article. 

The key step is the exponential decay of the Fisher information along the semigroup, which we deduce from the gradient estimate through the following variational characterisation of the Fisher information.

\begin{lemma}[Variational Characterisation of the Fisher Information]\label{lem:fisher_variational}
    Let $f$ be smooth and bounded with $\inf f > 0$. Then
    \begin{equ}[eq:fisher_variational]
        \CI(f) \eqdef \mathbb{E}_{\mu_{V}}\Big(\frac{\Gamma(f)}{f}\Big) = \sup_{a}\Big\{2\,\CE(f,a) - \mathbb{E}_{\mu_{V}}\big(f\,\Gamma(a)\big)\Big\}\;,
    \end{equ}
    where the supremum runs over smooth bounded $a$ and is attained at $a = \log f$.
\end{lemma}

\begin{proof}
    By the diffusion chain rule, $\Gamma(f,a) = f\,\Gamma(\log f, a)$, so completing the square gives, for every test function $a$,
    \begin{equ}
        2\,\Gamma(f,a) - f\,\Gamma(a) = f\,\Big(\Gamma(\log f) - \Gamma\big(a - \log f\big)\Big) \leqslant f\,\Gamma(\log f) = \frac{\Gamma(f)}{f}\;.
    \end{equ}
    Taking $\mathbb{E}_{\mu_{V}}$ and recalling $\CE(f,a) = \mathbb{E}_{\mu_{V}}\big(\Gamma(f,a)\big)$ from \eqref{eq:IbyP} shows that the supremum is bounded above by $\CI(f)$, with equality at $a = \log f$, which is an admissible test function since $f$ is bounded with $\inf f > 0$, so that $\log f$ is smooth and bounded.
\end{proof}

\begin{lemma}[Fisher Information Decay]\label{lem:fisher_decay}
    Assume that $\Gamma_{2} \geqslant \kappa\,\Gamma$, and let $f$ be smooth and bounded with $\inf f > 0$. Then
    \begin{equ}
        \CI(P_{s}f) \leqslant e^{-2\kappa s}\,\CI(f) \qquad \text{for all } s \geqslant 0\;.
    \end{equ}
\end{lemma}

\begin{proof}
    We use two consequences of the symmetry of $P_{s}$ on $L^{2}(\mu_{V})$. First, $\CE(P_{s}f,a) = \CE(f,P_{s}a)$, since $P_{s}$ commutes with $L$. Second, combined with the gradient estimate \eqref{eq:comm_gradient_estimate}, symmetry gives the weighted estimate
    \begin{equ}[eq:weighted_GE]
        \mathbb{E}_{\mu_{V}}\big(f\,\Gamma(P_{s}a)\big) \leqslant e^{-2\kappa s}\,\mathbb{E}_{\mu_{V}}\big(f\,P_{s}\Gamma(a)\big) = e^{-2\kappa s}\,\mathbb{E}_{\mu_{V}}\big((P_{s}f)\,\Gamma(a)\big)\;.
    \end{equ}
    Since $P_{s}$ is positivity preserving and $P_{s}\bone = \bone$, we have $\inf f \leqslant P_{s}f \leqslant \sup f$, so Lemma~\ref{lem:fisher_variational} applies to $P_{s}f$, giving
    \begin{equs}
        \CI(P_{s}f) &= \sup_{a}\Big\{2\,\CE(P_{s}f,a) - \mathbb{E}_{\mu_{V}}\big((P_{s}f)\,\Gamma(a)\big)\Big\} =\\
        &= \sup_{a}\Big\{2\,\CE(f,P_{s}a) - \mathbb{E}_{\mu_{V}}\big((P_{s}f)\,\Gamma(a)\big)\Big\} \leqslant\\
        &\leqslant \sup_{a}\Big\{2\,\CE(f,P_{s}a) - e^{2\kappa s}\,\mathbb{E}_{\mu_{V}}\big(f\,\Gamma(P_{s}a)\big)\Big\} =\\
        &= e^{-2\kappa s}\,\sup_{a}\Big\{2\,\CE\big(f, e^{2\kappa s}P_{s}a\big) - \mathbb{E}_{\mu_{V}}\Big(f\,\Gamma\big(e^{2\kappa s}P_{s}a\big)\Big)\Big\} \leqslant\\
        &\leqslant e^{-2\kappa s}\,\CI(f)\;,
    \end{equs}
    where the second line is the duality $\CE(P_{s}f,a) = \CE(f,P_{s}a)$ applied inside the supremum, the third uses \eqref{eq:weighted_GE}, the fourth uses that $\Gamma$ is quadratic, and the last that $e^{2\kappa s}P_{s}a$ is again an admissible test function, together with the upper bound in \eqref{eq:fisher_variational}.
\end{proof}

\begin{theorem}[Logarithmic Sobolev inequality for Langevin dynamics]\label{thm:LSI_Langevin}
    Let $\mu_V$ be above, and suppose that $\operatorname{Hess}(V) \geqslant \kappa > 0$.
    Then, for all smooth $f > 0$,
    \begin{equ}
        \operatorname{Ent}_{\mu_{V}}^{\Phi}(f) \leqslant \frac{1}{2\kappa}\,\mathbb{E}_{\mu_{V}}\Big(\frac{\Gamma(f)}{f}\Big)\;.
    \end{equ}
\end{theorem}

\begin{proof}
As we discussed in Section~\ref{ss:invariant_measures},  the semigroup $P_{t}$ is ergodic with respect to $\mu_{V}$. Since $\Phi''(u) = 1/u$, ergodicity gives
\begin{equ}[eq:comm_entropy_dissipation]
\operatorname{Ent}_{\mu_{V}}^{\Phi}(f) = -\int_{0}^{\infty}\partial_{s}\,\mathbb{E}_{\mu_{V}}\Phi(P_{s}f)\,\d s = \int_{0}^{\infty}\mathbb{E}_{\mu_{V}}\Big(\frac{\Gamma(P_{s}f)}{P_{s}f}\Big)\,\d s\;,
\end{equ}
where the last equality uses the integration-by-parts formula \eqref{eq:IbyP}, which is a consequence of the diffusion property. The right-hand side is $\int_{0}^{\infty}\CI(P_{s}f)\,\d s$ in the notation \eqref{eq:fisher_variational}. We may moreover assume that $f$ is bounded with $\inf f > 0$, the general case following by truncation and monotone convergence.
As $\operatorname{Hess}(V) \geqslant \kappa$ is equivalent to $\Gamma_{2} \geqslant \kappa\,\Gamma$, Lemma~\ref{lem:fisher_decay} gives $\CI(P_{s}f) \leqslant e^{-2\kappa s}\,\CI(f)$, and the claim follows upon integrating, since $\int_{0}^{\infty}e^{-2\kappa s}\,\d s = 1/(2\kappa)$.
\end{proof}

\begin{remark}
The feature of $\Phi(u) = u\log u$ used above is that $\Phi''(u) = 1/u$, so that the dissipation $\mathbb{E}_{\mu_{V}}\big(\Phi''(f)\,\Gamma(f)\big)$ equals $\CE(f,\log f)$; it is this logarithmic structure that underlies the variational characterisation \eqref{eq:fisher_variational}, in which the supremum is attained at $\log f$. 

Suppose now that one has the strengthened gradient estimate $\sqrt{\Gamma(P_{t}f)} \leqslant e^{-\kappa t}\,P_{t}\big(\sqrt{\Gamma(f)}\big)$, and note that the function $(u,v) \mapsto \Phi''(u)v^{2}$ is jointly convex precisely when $1/\Phi''$ is concave. For such $\Phi$, squaring the gradient estimate and applying Jensen's inequality for this function to the pair $\big(f,\sqrt{\Gamma(f)}\big)$ yields $\Phi''(P_{s}f)\,\Gamma(P_{s}f) \leqslant e^{-2\kappa s}\,P_{s}\big(\Phi''(f)\Gamma(f)\big)$. Since the derivative appearing in the integrand of \eqref{eq:comm_entropy_dissipation} is, for a general convex $\Phi$, given by $\partial_{s}\,\mathbb{E}_{\mu_{V}}\Phi(P_{s}f) = -\mathbb{E}_{\mu_{V}}\big(\Phi''(P_{s}f)\,\Gamma(P_{s}f)\big)$, bounding the right-hand side by the display above, using the invariance of $\mu_{V}$, and integrating over $s$ extends the theorem to this class of $\Phi$.

Our argument here is meant to motivate the free probabilistic proof of Section~\ref{ss:LSI_NC}: the characterisation \eqref{eq:fisher_variational} and the duality $\CE(P_{s}f,a) = \CE(f,P_{s}a)$ carry over to that setting.
On the other hand, while the weighted gradient estimate \eqref{eq:weighted_GE} is immediate here, since the weight commutes with the carr\'{e} du champ, it is replaced by the logarithmic-mean gradient estimate \eqref{eq:lmge} of Proposition~\ref{prop:tensor-bridge} in the free probabilistic setting.
\end{remark}

\subsection{Hypercontractivity}\label{ss:hypercontractivity}

\begin{definition}[Hypercontractivity]\label{def:hypercontractivity}
Suppose $1 < p \leqslant q < \infty$. We say that a semigroup $P_{t}$ is \emph{hypercontractive} from $L^{p}$ to $L^{q}$ if
$\|P_{t}\|_{L^{p} \to L^{q}} \leqslant 1$.
\end{definition}

Given the logarithmic Sobolev constant $C > 0$, we define \emph{Nelson's shortest time} by
\begin{equ}
t_{N}(p,q) \eqdef \frac{C}{2}\log\frac{q-1}{p-1}\;,
\end{equ}
and the associated time-dependent exponent
\begin{equ}
q(t) \eqdef e^{2t/C}(p-1) + 1\;.
\end{equ}
Note that $q(t) > p$ for $t > 0$, so $P_{t}$ maps $L^{p}$ into a \emph{strictly smaller} space.

\begin{theorem}[Equivalence of Hypercontractivity and LSI]\label{thm:HC_LSI}
    Suppose $(E, \Gamma, \mu)$ is a Markov triple. Then $\operatorname{LSI}(C)$ holds if and only if for all $1 < p \leqslant q < \infty$ and all $t \geqslant t_N(p,q)$
    \begin{equ}[eq:HC]
        \|P_{t}\|_{L^{p} \to L^{q}} = 1 \;.
    \end{equ}
\end{theorem}

\begin{corollary}\label{cor:eigenfunctions_Lq}
All eigenfunctions of $L$ belong to $L^{q}$ for every $2 \leqslant q < \infty$.
\end{corollary}

\begin{proof}[Proof of Corollary~\ref{cor:eigenfunctions_Lq}]
Suppose $Lf = -\lambda f$ with $\lambda > 0$. Then $P_{t}f = e^{-\lambda t}f$, and for $p = 2$ and all $t \geqslant t_{N}(2,q) = \frac{C}{2}\log(q-1)$,
\begin{equ}
\|P_{t}f\|_{q} = e^{-\lambda t}\|f\|_{q} \leqslant \|f\|_{2}\;.
\end{equ}
In particular, $\|f\|_{q} \leqslant (q-1)^{C\lambda/2}\|f\|_{2}$.
\end{proof}

\begin{proof}[Proof of Theorem~\ref{thm:HC_LSI}]
\emph{Hypercontractivity $\Rightarrow$ LSI.}
Suppose $f > 0$ is smooth. Define $F(t) \eqdef \|P_{t}f\|_{q(t)}$.
A computation gives
\begin{equs}
    \|P_{t}f\|_{q}^{q-1}\frac{\d }{\d t}\|P_{t}f\|_{q(t)} &= \frac{q'}{q}\bigg(\int_{E}(P_{t}f)^{q}\log P_{t}f\,\d \mu - \|P_{t}f\|_{q}^{q}\log\|P_{t}f\|_{q}\bigg) + \\
    &\qquad + \int_{E}(P_{t}f)^{q-1}LP_{t}f\,\d \mu\;.
\end{equs}
Evaluating at $t = 0$ with $p = 2$ (so $q(0) = 2$, $q'(0) = 2/C$), the hypercontractivity assumption gives $F'(0) \leqslant 0$, which yields $\operatorname{LSI}(C)$.

\emph{LSI $\Rightarrow$ Hypercontractivity.}
We show that $\|P_{t}f\|_{q(t)}$ is non-increasing.
Applying $\operatorname{LSI}(C)$ to the function $(P_{t}f)^{q/2}$ and using the chain rule for diffusion operators,
\begin{equ}
    \CE\big((P_{t}f)^{q/2}, (P_{t}f)^{q/2}\big) = \frac{q^{2}}{4}\int_{E}(P_{t}f)^{q-2}\,\Gamma(P_{t}f)\,\d \mu\;,
\end{equ}
we obtain
\begin{equ}
    \int_{E}(P_{t}f)^{q}\log P_{t}f\,\d \mu - \|P_{t}f\|_{q}^{q}\log\|P_{t}f\|_{q} \leqslant \frac{qC}{2}\int_{E}(P_{t}f)^{q-2}\,\Gamma(P_{t}f)\,\d \mu\;.
\end{equ}
By integration by parts and the chain rule,
\begin{equ}
    \int_{E}(P_{t}f)^{q-1}LP_{t}f\,\d \mu = -(q-1)\int_{E}(P_{t}f)^{q-2}\,\Gamma(P_{t}f)\,\d \mu\;.
\end{equ}
Combining with the derivative formula above, we obtain
\begin{equ}
    \|P_{t}f\|_{q}^{q-1}\frac{\d }{\d t}\|P_{t}f\|_{q(t)} \leqslant \Big(\frac{q'C}{2} - (q-1)\Big)\int_{E}(P_{t}f)^{q-2}\,\Gamma(P_{t}f)\,\d \mu = 0
\end{equ}
for $q(t) = e^{2t/C}(p-1) + 1$, since this is the solution to $q'/(q-1) = 2/C$.
Therefore $\|P_{t}f\|_{q(t)}$ is non-increasing, with $\|P_{0}f\|_{q(0)} = \|f\|_{p}$.
\end{proof}

\section{Free It\^{o} Processes and Formula}\label{sec:free_ito}

\subsection{Free Gaussian Functor and Semicircular Brownian Motion}\label{ss:free_gaussian}

We apply Voiculescu's free Gaussian functor to the real Hilbert space $\mfh = L^{2}(\R;\R^{m})$, yielding a corresponding free Fock space together with creation and annihilation operators $\alpha^{\dagger}(f), \alpha(g)$ for all $f, g \in \mfh$. The \emph{free Gaussian field} is defined as \begin{equ} \phi(f) \eqdef \alpha^{\dagger}(f) + \alpha(f)\;. \end{equ}

Let $\CM_X$ denote the von Neumann algebra generated by $\{\phi(f)\mid f\in\mfh\}$, equipped with the vacuum state $\rho_X(a)\eqdef\Braket{\Omega,a\Omega}$, where $\Omega$ is the vacuum vector. The state $\rho_X$ is faithful and tracial.

For $t\geqslant0$ and $i\in[m]$, set $X_t^i\eqdef\phi(\mathbf1_{[0,t]}\otimes e_i)$, where $(e_i)_{i\in[m]}$ is the standard basis of $\RR^m$, and let $\CM_{X,t}$ be generated by $\{X_s^i\mid 0\leqslant s\leqslant t,\ i\in[m]\}$. Then $X_t=(X_t^1,\dots,X_t^m)$ is an $m$-dimensional semicircular Brownian motion in $(\CM_X,(\CM_{X,t})_{t\geqslant0},\rho_X)$.

\subsection{Free It\^{o} Formula and Generator}\label{ss:free_ito_formula}

Given a self-adjoint potential $V \in \Poly$, we study the system of free stochastic differential equations 
\begin{equ}[eq:SDEs] 
    d Y_{t} = - DV(Y_{t})\,d \! t + \sqrt{2}\,d X_{t} \;, 
\end{equ} 
inside a tracial von Neumann algebra $(\hat\CM,\hat\rho)$ containing $(\CM_X,\rho_X)$.

We assume that $\hat\rho$ is a faithful tracial state on $\hat{\CM}$, that $Y_0=(Y_0^1,\dots,Y_0^m)$ is freely independent of $(X_t)_{t\geqslant0}$ and has law $\omega$, and that \eqref{eq:SDEs} has a global, adapted solution. Let $\digamma_t\colon\Poly\to\hat\CM$ be the evaluation morphism determined by $\digamma_t(\indet^i)=Y_t^i$, and identify $\Poly$ with $\digamma_0(\Poly)$. We write 
\begin{equ} 
    \CM \eqdef W^*(Y_0) = \overline{\Poly}^{\,w^*}\;. 
\end{equ}
For $a \in \Poly$, set 
\begin{equ} 
    P_t a \eqdef \E_0[a(Y_t)] \; , 
\end{equ} 
where $\E_0 \colon \hat\CM \to \CM $ is the $\hat\rho$-preserving conditional expectation. 
\begin{assumption}[Invariance]\label{ass:standing_invariance} 
    We assume that $\hat\rho$ is an invariant state for the dynamics, i.e.\ that for all $t \geqslant 0$
    \begin{equ}
        \omega = \digamma_t^* \hat \rho \; . 
    \end{equ}  
\end{assumption}
Under the preceding hypotheses, the Markov property makes $(P_t)_{t\geqslant0}$ a semigroup, while invariance supplies its trace-preserving extensions.
\begin{proposition}\label{prop:basic_cont}
    $(P_t)_{t \geqslant 0}$ extends to a unit preserving semigroup of contractions on $\CL^p(\omega)$ for all $p \in [1,\infty]$, i.e.\
    \begin{equ}[eq:Pcont] 
        \|P_t a\|_{\CL^p(\omega)} \leqslant \|a\|_{\CL^p(\omega)} \; ,
    \end{equ} 
    and it is completely positive. Furthermore, it is also strongly continuous on $\CL^p(\omega)$ for $p \in [1,\infty)$.
\end{proposition}
\begin{proof}
Pathwise uniqueness and freeness of the future increments give the Markov property as in \cite[Section~3.2]{BianeSpeicher2001}. The associated reduced evolution is therefore an ultraweakly continuous, unital completely positive semigroup preserving $\omega$ and thus an $\CL^p(\omega)$ contraction for all $p \in [1,\infty]$; see \cite[Theorem~2.1]{AccardiFrigerioLewis1982}.

The norm continuity in time of the solution trajectories follows as in \cite[Theorem~3.1]{BianeSpeicher2001}. Hence, for $a\in\Poly$,
    \begin{equ}
        \|P_ta-P_sa\|_{\CL^p(\omega)} \leqslant \|a(Y_t)-a(Y_s)\|_{\CL^p(\hat\rho)}   \xrightarrow{t \to s} 0 \; .
    \end{equ}
Contractivity and density of $\Poly$ then give strong continuity on  $\CL^p(\omega)$ for $1\leqslant p<\infty$. The case $p=2$, followed by density of $\CL^2(\omega)$ in $\CL^1(\omega)$, also gives strong (in time) ultraweak continuity on $\CM$.
\end{proof}

 With the normalisation of the free Gaussian field fixed above, the free It\^{o} formula gives, for $a\in\Poly$, 
    \begin{equs}[eq:free_ito] 
        a(Y_{t}) &= a(Y_{0}) + \sum_{i=1}^{m}\int_{0}^{t}(\partial_{i}a)(Y_{s})\tw \d Y_{s}^{i} + 2\sum_{i=1}^{m}\int_{0}^{t}\eta\big(\partial_{i}^{2}a\big)(Y_{s})\,\d s = \\
        &= a(Y_{0}) + \sqrt{2}\sum_{i=1}^{m}\int_{0}^{t}(\partial_{i}a)(Y_{s})\tw \d X_{s}^{i} + \\
        &\qquad + \sum_{i=1}^{m} \int_{0}^{t}\bigg( - (\partial_{i}a)(Y_{s})\tw D_{i}V(Y_{s}) + 2\eta(\partial_{i}^{2}a)(Y_{s})\bigg) \, \d s \; , 
    \end{equs} 
where $\eta \colon \hat{\mathcal{M}}^{\otimes 3} \rightarrow \hat{\mathcal{M}}$ maps $A \otimes B \otimes C$ to $\hat{\rho}(B)AC$, cf.\ \cite{nikitopoulos2022ito}. Taking the conditional expectation in \eqref{eq:free_ito} identifies the generator on its polynomial core.

Pulling the generator back to the polynomial core gives 
\begin{equ}[eq:generator] 
    \Lin_{V}\, a = -\sum_{i=1}^{m}(\partial_{i}a)\tw D_{i}V + 2\Delta_{\omega}a\;. 
\end{equ} 
Here we have defined the \emph{reduced Laplacian} $\Delta_{\omega} \colon \Poly \rightarrow \Poly$ by 
\begin{equ} 
    \Delta_{\omega}a \eqdef \frac{1}{2} \bar{\eta}\,\Delta a\;, 
\end{equ} 
where $\bar{\eta}(a\otimes b\otimes c)\eqdef\omega(b)ac$. Since $\Delta=2\sum_i(\partial_i\otimes\Id)\partial_i$, the finite-variation term in \eqref{eq:free_ito} is exactly $2\Delta_\omega a$. Notice that $\Lin_V\Poly\subset\Poly$, but $P_t\Poly$ need not be contained in $\Poly$ when $V$ is non-quadratic.

The normal Markov maps $P_t$ extend consistently to contractions on $\CL^p(\omega)$ for $1\leqslant p<\infty$. 

From this point onwards, we make the following standing assumption.

\begin{assumption}[Standing Symmetry and Core Assumption]\label{ass:standing_symmetry} 
    The semigroup $(P_t)_{t \geqslant 0}$ is $\omega$-symmetric: for all $a,b \in \CL^2(\omega)$ and $t \geqslant 0$
    \begin{equ}[eq:standing_symmetry] 
        \Braket{a,P_tb}_{\CL^2(\omega)} = \Braket{P_ta,b}_{\CL^2(\omega)} \; . 
    \end{equ}
    We will denote by $\Lin_V$ the non-positive self-adjoint generator of $P_t$ in $\CL^2(\omega)$ and by $(\Lin_V)_1$ the generator in $\CL^1(\omega)$. 
    
    Moreover, we assume that $\Poly$ is an operator core for the domain $D\left((-\Lin_V)^{3/2}\right)$ of $(-\Lin_V)^{3/2}$. 
\end{assumption}

\begin{remark}
    Since $P_t\bone=\bone$, symmetry implies invariance: 
    \begin{equ}[eq:standing_invariance] 
        \omega(P_ta) = \Braket{\bone,P_ta}_{\CL^2(\omega)} = \Braket{P_t\bone,a}_{\CL^2(\omega)} = \omega(a) \; . 
    \end{equ}
    In particular, $\hat\rho(a(Y_t))=\omega(P_ta)=\omega(a)$, so $\digamma_t^*\hat\rho=\omega$. Thus, stationarity can be seen as another manifestation of the standing symmetry assumption. The generator $\Lin_V$ is self-adjoint and non-positive on $\CL^2(\omega)$, with $\Lin_V\bone=0$. 
\end{remark}

\begin{remark}[Normalisation and Schwinger--Dyson Identity]\label{rem:normalisation_SD}
    The factor $\sqrt{2}$ in the noise in \eqref{eq:SDEs} makes the raw free difference quotient energy-normalised. Moreover, symmetry and $\Lin_V\indet^i=-D_iV$ imply 
    \begin{equ}[eq:standing_SD]
        (\omega\otimes\omega)(\partial_i a) = \omega(D_iV\,a)\;, 
    \end{equ}
    for $a\in\Poly$. Thus the conjugate variable is $D_iV$. Equivalently, the present dynamics is the factor-two time-rescaling of the unit-noise SDE with drift $-\frac12DV$.
\end{remark}

\begin{remark}\label{rem:DomainAss}
    The density of $\Poly$ in $D( (-\Lin_V)^{\frac{3}{2}} )$ is motivated by \cite[Lemma~19(i),(iii)]{AIHPB_2014__50_4_1404_0}. This result ensures that, for $f\in D\bigl((-\Lin_V)^{3/2}\bigr)$,
    \begin{equ}[eq:closed_gradient_generator_domain]
        \overline\nabla_i f\in D((-\Lin_V)^e)   \; , 
    \end{equ}
    Here $(-\Lin_V)^{e}$ denotes the enveloping extension of the generator, introduced in Section~\ref{sec:nc_gamma} below.
    In particular, $\overline{\nabla}_i \colon D\bigl((-\Lin_V)^{3/2}\bigr) \to  D((-\Lin_V)^e)$ is continuous for all $i \in [m]$.

    We may apply this lemma as the Schwinger--Dyson identity gives $\partial_j^*(\bone\otimes\bone)=D_jV\in\Poly$, while $\nabla_iD_jV\in\Poly^e\subset\CM^e$, and thus the hypotheses of the lemma are satisfied.

    Furthermore, this assumption ensures that $\Poly$ is dense in $D( (-\Lin_V)^{\alpha})$ for all $\alpha \in [0,\frac{3}{2}]$ by the H\"older inequality.

\end{remark}

We record a useful fact about the generator. 

\begin{lemma}[Preservation of the Involution]\label{lem:generator_star_preserving} 
    If $V = V^*$, then, for all $a \in \Poly$, 
    \begin{equ} 
        (\Lin_V a)^* = \Lin_V(a^*) \; . 
    \end{equ}
\end{lemma} 
\begin{proof} 
    Recall the convention~\eqref{eq:partial_star} from Section~\ref{sec:notation}: writing $\partial_{i}a = \sum a^{(1)} \otimes a^{(2)}$ with the summation left implicit (and the dependence on $i$ suppressed), it reads $\partial_{i}(a^{\ast}) = \sum (a^{(2)})^{\ast} \otimes (a^{(1)})^{\ast}$.

    We first observe that the cyclic derivative commutes with the $\ast$-operation. Since $D_{i} = \mcb{m} \circ \iota \circ \partial_{i}$ and $\iota$ acts componentwise, \eqref{eq:partial_star} gives 
    \begin{equs} 
        D_{i}(a^{\ast}) &= \mcb{m}\,\iota\,\partial_{i}(a^{\ast}) = \mcb{m}\Big(\textstyle\sum (a^{(1)})^{\ast} \otimes (a^{(2)})^{\ast}\Big) = \\
        &= \sum (a^{(1)})^{\ast}(a^{(2)})^{\ast} = \Big(\textstyle\sum a^{(2)} a^{(1)}\Big)^{\ast} = (D_{i}a)^{\ast}\;. 
    \end{equs} 
    In particular, since $V = V^{\ast}$, each $D_{i}V$ is self-adjoint.

    \emph{First-Order Term.} Using the definition $(x \otimes y)\tw c = xcy$, the self-adjointness of $D_{i}V$, and then \eqref{eq:partial_star}, 
    \begin{equs} 
        \Big(\sum_{i=1}^{m}(\partial_{i}a)\tw D_{i}V\Big)^{\ast} &= \sum_{i=1}^{m}\sum \big(a^{(1)}\,(D_{i}V)\,a^{(2)}\big)^{\ast}= \\
        &= \sum_{i=1}^{m}\sum (a^{(2)})^{\ast}\,(D_{i}V)\,(a^{(1)})^{\ast} = \\
        &= \sum_{i=1}^{m}\big(\partial_{i}(a^{\ast})\big)\tw D_{i}V\;. 
    \end{equs}

    \emph{Second-Order Term.} Define the antilinear involution $\sharp$ on $\Poly^{\otimes 3}$ by $(x \otimes y \otimes z)^{\sharp} \eqdef z^{\ast} \otimes y^{\ast} \otimes x^{\ast}$. We claim that
    \begin{equ}[eq:Delta_star] 
        \Delta(a^{\ast}) = (\Delta a)^{\sharp}\;.
    \end{equ} 
    Both sides are antilinear in $a$, so it suffices to verify \eqref{eq:Delta_star} on monomials $a = \indet^{w}$, $w \in \Words$. Write $\overleftarrow{w}$ for the reversal of $w$; since the indeterminates are self-adjoint, $(\indet^{w})^{\ast} = \indet^{\overleftarrow{w}}$. Recalling $\Delta = 2\sum_{i}(\partial_{i} \otimes \Id)\partial_{i}$, one has 
    \begin{equ} 
        \Delta \indet^{w} = 2 \sum_{k=1}^{m}\ \sum_{w = u\,k\,v\,k\,s} \indet^{u} \otimes \indet^{v} \otimes \indet^{s}\;, 
    \end{equ} 
    the inner sum running over factorizations of $w$ that single out two occurrences of the letter $k$. Reversal sends a factorization $w = u\,k\,v\,k\,s$ to $\overleftarrow{w} = \overleftarrow{s}\,k\,\overleftarrow{v}\,k\,\overleftarrow{u}$, and $(u,v,s) \mapsto (\overleftarrow{s},\overleftarrow{v},\overleftarrow{u})$ is a bijection between the corresponding index sets. Hence, 
    \begin{equ} 
        \Delta\big((\indet^{w})^{\ast}\big) = \Delta \indet^{\overleftarrow{w}} = 2 \sum_{k=1}^{m}\sum_{w = u\,k\,v\,k\,s} \indet^{\overleftarrow{s}} \otimes \indet^{\overleftarrow{v}} \otimes \indet^{\overleftarrow{u}} = \Big(\Delta \indet^{w}\Big)^{\sharp}\;, 
    \end{equ} 
    which is \eqref{eq:Delta_star}. Next, since $\omega$ is a state we have $\omega(b^{\ast}) = \overline{\omega(b)}$, so for any $x \otimes y \otimes z \in \Poly^{\otimes 3}$, 
    \begin{equ} 
        \bar{\eta}\big((x \otimes y \otimes z)^{\sharp}\big) = \omega(y^{\ast})\,z^{\ast}x^{\ast} = \overline{\omega(y)}\,(xz)^{\ast} = \big(\omega(y)\,xz\big)^{\ast} = \big(\bar{\eta}(x \otimes y \otimes z)\big)^{\ast} \; . 
    \end{equ} 
    As $\bar{\eta} \circ \sharp$ and $\ast \circ \bar{\eta}$ are both antilinear, this extends to $\bar{\eta}(W^{\sharp}) = (\bar{\eta}W)^{\ast}$ for all $W \in \Poly^{\otimes 3}$. Combined with \eqref{eq:Delta_star}, 
    \begin{equ} 
        \Delta_{\omega}(a^{\ast}) = \tfrac{1}{2} \bar{\eta}\,\Delta(a^{\ast}) = \tfrac{1}{2} \bar{\eta}\big((\Delta a)^{\sharp}\big) = \big( \tfrac{1}{2} \bar{\eta}\,\Delta a\big)^{\ast} = (\Delta_{\omega}a)^{\ast}\;. 
    \end{equ}
    Adding the two contributions and using \eqref{eq:generator}, 
    \begin{equs} 
        (\Lin_V a)^{\ast} &= -\Big(\sum_{i=1}^{m}(\partial_{i}a)\tw D_{i}V\Big)^{\ast} + 2(\Delta_{\omega}a)^{\ast} = \\
        &= -\sum_{i=1}^{m}\big(\partial_{i}(a^{\ast})\big)\tw D_{i}V + 2\Delta_{\omega}(a^{\ast}) = \Lin_V(a^{\ast})\;. 
    \end{equs} 
\end{proof}

\section{Noncommutative Gamma Calculus}\label{sec:nc_gamma}

\subsection{Opposite Algebras} Recall that an algebra $A$ is an underlying vector space (which we denote by $\underline{A}$) equipped with a multiplication map $\mcb{m} \colon \underline{A} \times \underline{A} \to \underline{A}$. Given such an algebra $A$, we can define the \emph{opposite algebra} $\Aop$ to be the same underlying vector space $\underline{A}$ but equipped with the (opposite) multiplication map $\mop \colon \underline{A} \times \underline{A} \to \underline{A}$ with 
\begin{equ} 
    \mop(a,b) \eqdef \mcb{m}(b,a) \;. 
\end{equ}

For any $*$-algebra $A$, we are also given a conjugate-linear involution $*\colon \underline{A} \to \underline{A}$. We then clearly have
\begin{equs} 
    \mop \circ (* \otimes *) = * \circ \mcb{m} \; . 
\end{equs}

We also remark that both $A$ and $\Aop$ are naturally both $A$- and $\Aop$-bimodules.

Let $\bone^{\mathrm{op}} \colon A \to \Aop$ be the identity map on the underlying vector space $\underline{A}$. This is not an algebra homomorphism, but it commutes with $*$ and is an $A$- and $\Aop$-bimodule morphism. We denote its inverse also by $\bone^{\mathrm{op}}$.

We also introduce the algebraic enveloping algebra $A^e = \Aop \otimes A$. The multiplication on $A^{e}$ is component-wise. Denoting the product on $A^{e}$ by $\cdot$, and adopting the notation convention where we sometimes suppress $\mcb{m}$, that is we write $\mcb{m}(a,b) = ab$, we then have 
\begin{equ} 
    (a\otimes b) \cdot (c \otimes d) = \mop(a,c) \otimes \mcb{m}(b,d) = ca \otimes bd \;. 
\end{equ}
$A^{e}$ is isomorphic to $A \otimes A$ as a vector space, but not as an algebra when $A \otimes A$ is equipped with its standard component-wise product. The product $\cdot$ is instead a type of ``insertion'' product.

Finally, we let $\iota^{e}\colon A^{e} \to A^{e}$ denote the involution given by 
\begin{equ} 
    \iota^{e}( a \otimes b ) \eqdef \bop(b) \otimes \bop(a) \; . 
\end{equ}

We view $A^{e}$ as an  $A$-bimodule by setting, for $a,b \in A$ and $c \otimes d \in A^{e} = A^{\op} \otimes A$, 
\begin{equ} 
    a \cdot (c \otimes d) \cdot b \eqdef ac \otimes db  \; . 
\end{equ} 
Note that with this $A$-action, $A^{e}$ is generically not an $A$-algebra. We also note that we can similarly make $(A^{\op})^{e} = (A^{\op})^{\op} \otimes A^{\op} = A \otimes A^{\op}$ an $A^{\op}$-bimodule and we use the notation $\cdot$ again for this action.

Specialising to $A=\Poly$ or $A=\pop$, we use $\partial_i$ for the free difference quotient on either algebra. With the noise normalisation in \eqref{eq:SDEs}, the raw free difference quotient already has the Dirichlet-form normalisation, so we set 
\begin{equ}[eq:raw_and_energy_gradient] 
    \nabla_i \eqdef (\bone^{\op}\otimes\bone)\partial_i\;,\qquad  \nabla\eqdef(\nabla_1,\dots,\nabla_m)\;. 
\end{equ} 
$\nabla$ is a derivation; in particular,
\begin{equ} 
    \nabla_i(ab) = (\nabla_i a)\cdot b + a\cdot(\nabla_i b)\;. 
\end{equ}
We extend $\Lin_V$ to $\pop$ by setting $ \Lin_V^{\op}   \eqdef\bop\circ\Lin_V\circ\bop$.

At the von Neumann algebra level, we define (by a slight abuse of notation)
\begin{equ}[eq:completed_enveloping_algebra] 
    \CM^e \eqdef \CM^{\op}\vnotimes\CM\;,\qquad \omega^e \eqdef \omega^{\op}\otimes\omega\;. 
\end{equ} 
The induced semigroup and generator are 
\begin{equ}[eq:tensor_semigroup] 
    P_t^e \eqdef P_t^{\op}\otimes P_t\;,\qquad \Lin_V^e \eqdef \Lin_V^{\op}\otimes\Id+\Id\otimes\Lin_V \; , 
\end{equ}
 where $P_t^{\op}(a^{\op})\eqdef P_t(a)^{\op}$. The spatial tensor product of normal completely positive maps makes $P_t^e$ normal, unital and completely positive. Moreover, the standing symmetry assumption gives 
 \begin{equ}[eq:tensor_semigroup_L2_symmetry] 
    \Braket{A,P_t^eB}_{\CL^2(\omega^e)} = \Braket{P_t^eA,B}_{\CL^2(\omega^e)} \; , 
\end{equ}
 first for simple tensors and then by density. Thus, $P_t^e$ is self-adjoint on $\CL^2(\omega^e)$. Furthermore, since $P_t^e$ is a contraction semigroup for all $p \in [1,\infty]$, \eqref{eq:tensor_semigroup_L2_symmetry} also implies that for all $A \in \CL^1(\omega^e)$, $B \in \CM^e$ and $t \geqslant 0$
 \begin{equ}[eq:tensor_semigroup_L1_symmetry]
    \omega^e( A P^e_t B ) = \omega^e(B P^e_t A) \; .
 \end{equ}

 Finally, by spectral calculus we have 
 \begin{equ}
    D(\Lin_V^e) = D(\Lin_V^{\op} \otimes \Id ) \cap D(\Id \otimes \Lin_V)\;.
 \end{equ}

\subsection{Noncommutative Carr\'{e} du Champ}\label{ss:nc_cdc}

Motivated by the commutative Definition~\ref{def:comm_gamma}, define, for $a,b\in\Poly$, 
\begin{equ}[eq:nc_gamma1] 
    \Gamma(a,b) \eqdef \frac{1}{2}\Big(\Lin_V(a^{\ast}b) - a^{\ast}\Lin_Vb - (\Lin_Va)^{\ast}b\Big)\;. 
\end{equ}
As usual, write $\Gamma(a)\eqdef\Gamma(a,a)$.

\begin{lemma}[Lifted Carr\'{e} du Champ]\label{lem:Gamma_properties} 
    For $a,b\in\Poly$, \begin{equ}[eq:Gamma_Laplacian] \Gamma(a,b) = \Delta_\omega(a^*b)-a^*\Delta_\omega b-(\Delta_\omega a)^*b\;. \end{equ} Moreover, if 
    \begin{equ}[eq:nc_gamma2] 
        \tilde\Gamma(a,b) \eqdef (\nabla a)^*\cdot\nabla b = \sum_{i=1}^m(\nabla_i a)^*\cdot\nabla_i b \in\Poly^e \; , 
    \end{equ}
    then 
    \begin{equ}[eq:gamma_lift] 
        \Gamma(a,b) = (\omega^{\op}\otimes\Id)\bigl(\tilde\Gamma(a,b)\bigr) \; . 
    \end{equ}
    In particular, $\tilde\Gamma(a)\geqslant0$ in $\Poly^e$ and $\Gamma(a)\geqslant0$ in $\Poly$. 
\end{lemma}

\begin{proof} 
    The drift in \eqref{eq:generator} is a derivation and therefore cancels from \eqref{eq:nc_gamma1}, proving \eqref{eq:Gamma_Laplacian}. Since $\Delta=2\sum_i(\partial_i\otimes\Id)\partial_i$, two applications of the Leibniz rule give 
    \begin{equ}[eq:Delta_product_difference] 
        \Delta(f^*g)-f^*\Delta g-\Delta(f^*)g = 2\sum_{i=1}^m(\partial_i f^*)\amalg(\partial_i g)\;. 
    \end{equ}
     If $\partial_i f=\sum_k a_k\otimes b_k$ and $\partial_i g=\sum_l c_l\otimes d_l$, then traciality gives 
     \begin{equ}
         \sum_{k,l}\bar\eta(a_k^*\otimes b_k^*c_l\otimes d_l) = \sum_{k,l}\omega(b_k^*c_l)a_k^*d_l = (\omega^{\op}\otimes\Id)\bigl((\nabla_i f)^*\cdot\nabla_i g\bigr)\;. 
    \end{equ} 
    Combining this identity with \eqref{eq:Gamma_Laplacian} and \eqref{eq:Delta_product_difference} proves \eqref{eq:gamma_lift}. Positivity follows from \eqref{eq:nc_gamma2} and positivity of the slice map $\omega^{\op}\otimes\Id$.
 \end{proof}

\begin{proposition}[Closed Dirichlet Form]\label{prop:DirichletForm} 
    On $\Poly$, define the Dirichlet form 
    \begin{equ}
        \CE_0(a,b) \eqdef \omega\bigl(\Gamma(a,b)\bigr)\;. 
    \end{equ}
    Then 
    \begin{equ}[eq:energy_gradient_identity] 
        \CE_0(a,b) = \omega^e\bigl(\tilde\Gamma(a,b)\bigr) = \Braket{\nabla a,\nabla b}_{\CL^2(\omega^e)^m} = -\omega(a^*\Lin_Vb)\;. 
    \end{equ}
    Consequently, $\nabla\colon\Poly\to\CL^2(\omega^e)^m$ is closable. If $\overline\nabla$ denotes its closure, then the closed Dirichlet form $\CE$ associated with $-\Lin_V$ is 
    \begin{equ}[eq:closed_gradient_form] 
        \CE(a,b)=\Braket{\overline\nabla a,\overline\nabla b}_{\CL^2(\omega^e)^m}\;,
    \end{equ}
    with $D(\CE)=D\bigl((-\Lin_V)^{1/2}\bigr)=D(\overline\nabla)$, and 
    \begin{equ}[eq:generator_gradient_factorisation] 
        -\Lin_V = \overline\nabla^{\,*}\overline\nabla \; .
    \end{equ} 
    Furthermore, $\CE$ satisfies for all $a \in D(\CE)_+$ and $n \geqslant 0$
    \begin{equ}[eq:DirProp]
        \CE(a \wedge n) \leqslant \CE(a) \; . 
    \end{equ}
    Finally, $P_t D(\CE) \subset D(\CE)$. 
\end{proposition}

\begin{proof} 
    The first two equalities in \eqref{eq:energy_gradient_identity} follow from Lemma~\ref{lem:Gamma_properties}. Invariance gives $\omega(\Lin_V(a^*b))=0$, while symmetry gives $\omega((\Lin_Va)^*b)=\omega(a^*\Lin_Vb)$. Applying these identities to \eqref{eq:nc_gamma1} proves the last equality. 
    
    Since $-\Lin_V$ is non-negative and self-adjoint, the quadratic form it defines is closed on $D((-\Lin_V)^{1/2})$. The operator-core assumption implies that $\Poly$ is also a form core. Hence the closure of $\CE_0$ is this form, which proves closability, and \eqref{eq:generator_gradient_factorisation}. 

    Since $P_t$ is a symmetric completely positive Markov semigroup, its associated form $\CE$ is completely Dirichlet. This proves \eqref{eq:DirProp}; see \cite[Theorems~3.2,~4.2]{AlbeverioHoegh-Krohn1977a}.

    Finally, by functional calculus we have that $P_t D\bigl((-\Lin_V)^{1/2}\bigr) \subset D\bigl((-\Lin_V)^{1/2}\bigr)$ which proves the last assertion.
\end{proof}

\begin{remark}\label{rem:twochoices}
     In the commutative case, the expressions \eqref{eq:nc_gamma1} and \eqref{eq:nc_gamma2} coincide. In the noncommutative setting, they differ: \eqref{eq:nc_gamma1} takes values in $\Poly$ while the ``lifted'' $\tilde{\Gamma}$ in \eqref{eq:nc_gamma2} takes values in $\Poly^{\otimes 2}$.

    As a consequence, when we turn to second-order objects such as the iterated carr\'{e} du champ, there will be both ``unlifted and lifted'' versions that are related by applying $\omega$ to one factor. 
\end{remark}

Proposition~\ref{prop:DirichletForm} now allows us to extend the quadratic carr\'e du champ maps. In particular, we obtain
\begin{equ}
    \widetilde{\Gamma} \colon D(\CE)  \longrightarrow \CL^1(\omega^e)_+  \; , \qquad \Gamma \colon D(\CE) \longrightarrow \CL^1(\omega)_+ \; .  
\end{equ}
Their polarisations are continuous bilinear forms with respect to the graph norm.

\begin{remark}
    By graph norm continuity and the density of $\Poly$ in the domains w.r.t.\ this norm it follows that it is sufficient to check the gradient estimates only on $\Poly$.
\end{remark}

\subsection{Two Formulations of the Gradient Estimate}\label{ss:nc_gradient_est}

Our next goal is to identify an analogue of the $\Gamma_{2}$ condition in the free probability setting.
Recall that in the commutative case, the Bakry--\'{E}mery condition
$\Gamma_{2}(f) \geqslant K\,\Gamma(f)$ is equivalent to the gradient estimate: for all $t \geqslant 0$,
\begin{equ}
    \Gamma(P_{t}f) = |\nabla(P_{t}f)|^{2} \leqslant e^{-2Kt}\, P_{t}\|\nabla f\|^{2} = e^{-2Kt}\, P_{t}\Gamma(f)\;.
\end{equ}
In the noncommutative case, the first and third equalities above do not hold (see Remark~\ref{rem:twochoices}), and we obtain two distinct formulations of the gradient estimate.

\begin{definition}[Weak and Strong Gradient Estimates]\label{def:NCGD} 
    Let $K > 0$. We say that the semigroup $(P_t)_t$ satisfies the \emph{weak (noncommutative) gradient estimate} $\mathrm{GE}_{\mathrm w}(K,\infty)$ if 
    \begin{equ}[eq:gradientestimate1]
        \Gamma(P_ta) \leqslant e^{-2Kt}P_t\Gamma(a) \; , 
    \end{equ} 
    in  $\CL^1(\omega)_+$, for $t \geqslant 0$ and every $a \in D(\overline{\nabla})$. It satisfies the \emph{strong gradient estimate} $\mathrm{GE}_{\mathrm s}(K,\infty)$ if 
    \begin{equ}[eq:gradientestimate2] 
        \tilde\Gamma(P_ta) \leqslant e^{-2Kt}P_t^e\bigl(\tilde\Gamma(a)\bigr) 
    \end{equ} 
    in $\CL^1(\omega^e)_+$. 
\end{definition}

The strong estimate implies the weak one. Indeed, applying the normal positive slice map $\omega^{\op}\vnotimes\Id$ and using \eqref{eq:standing_invariance} gives 
\begin{equ} 
    (\omega^{\op}\vnotimes\Id)P_t^e = P_t(\omega^{\op}\vnotimes\Id)\;, 
\end{equ}
and the conclusion follows from \eqref{eq:gamma_lift}.

\subsection{Auxiliary Operators}\label{ss:auxiliary_ops}

Given a linear map $A\colon\Poly\to\Poly$, define its enveloping-algebra lift by 
\begin{equ} 
    A^e \eqdef A^{\op}\otimes\Id+\Id\otimes A\;. 
\end{equ}
The two iterated carr\'{e} du champ operators are obtained by differentiating 
\begin{equ} 
    \Psi_K(s) \eqdef e^{-2Ks}P_s\Gamma(P_{t-s}f)\;,\qquad \tilde\Psi_K(s) \eqdef e^{-2Ks}P_s^e\tilde\Gamma(P_{t-s}f)\;. 
\end{equ} 
Under suitable core regularity, non-negativity of their derivatives yields \eqref{eq:gradientestimate1} and \eqref{eq:gradientestimate2}, respectively. 

\subsection{Product Formulae and Hessians}\label{ss:product_formulae}

We first record the product formulae for $\Lin_V$ and $\Lin_V^\op$ acting on products of polynomials. For $f, g \in \Poly$, 
\begin{equ}[eq:productL] 
    \Lin_V(f^{\ast}g) = (\Lin_Vf)^{\ast}g + f^{\ast}\Lin_Vg + 2(\omega^{\op}\otimes\Id)\bigl((\nabla f)^{\ast}\cdot\nabla g\bigr)\;. 
\end{equ}
An analogous formula also holds for the opposite algebra $\pop$. For the sake of clarity, we will distinguish in this derivation the $\Lin_V^{\op} \colon \pop \to \pop$ from $\Lin_V \colon \Poly \to \Poly$ although the two maps are the same on the underlying vector space. For $f^\op,g^\op \in \pop$, and writing $\cdot_{\op}$ for the product $\mcb{m}^{\op}$ in $\pop$, we have 
\begin{equs}[eq:productLop] 
    \Lin_V^{\op}\left( (f^\op)^* \cdot_{\op} g^\op\right) &= \bop\left( \Lin_V( g f^*) \right) = \\
    & = \bop\left( (\Lin_V g) f^* + g (\Lin_V f)^* + 2(\omega \otimes \Id_{\Poly}) \left[ \nabla g \cdot (\nabla f)^*\right] \right) = \\
    & =(f^\op)^* \cdot_{\op} (\Lin_V^\op g^\op) + (\Lin_V^\op f^\op)^* \cdot_{\op} g^\op +\\
    & \qquad \qquad\qquad  + 2(\omega \otimes \Id_{\pop}) \left[ (\nabla f^\op)^* \cdot_{\op} (\nabla g^\op) \right]\;. 
\end{equs} 
Here the second equality uses the product formula \eqref{eq:productL} together with Lemma~\ref{lem:generator_star_preserving}.

In terms of their domains, the product formula directly leads to the following corollary.
\begin{corollary}\label{cor:ProdDomain}
    Multiplication is a continuous map $D( \Lin_V )\times D( \Lin_V ) \to D( (\Lin_V)_1 )$ in terms of their respective graph norms. 
\end{corollary}
\begin{proof}
    This follows immediately for the first two terms on the right-hand side of \eqref{eq:productLop} by H\"older's inequality. The inclusion $D(\Lin_V)\subset D(\overline{\nabla})=D((-\Lin_V)^{1/2})$ and another application of H\"older's inequality handle the last term.
\end{proof}

\begin{lemma}\label{lem:product_formulae} 
    For $a \otimes b, c \otimes d \in \CP^{e}$, 
    \begin{equs}[eq:product_IdL] 
        (\Id_{\pop} \otimes \Lin_V)\big[(a \otimes b)^{\ast} \cdot (c \otimes d)\big] &= \big[(\Id_{\pop} \otimes \Lin_V)(a \otimes b)^{\ast}\big] \cdot (c \otimes d) + \\
        &\qquad + (a \otimes b)^{\ast} \cdot \big[(\Id_{\pop} \otimes \Lin_V)(c \otimes d)\big]+\\
        &\qquad + 2(\Id_{\pop} \otimes \omega^{\op} \otimes \Id_{\Poly})\big[ca^{\ast} \otimes \big((\nabla b)^{\ast} \cdot \nabla d\big)\big]\;. 
    \end{equs} 
    Furthermore, 
    \begin{equs}[eq:product_L2] 
        \Lin_V^e\big[(a \otimes b)^{\ast} \cdot (c \otimes d)\big] &= \big[\Lin_V^e(a \otimes b)^{\ast}\big] \cdot (c \otimes d) + (a \otimes b)^{\ast} \cdot \big[\Lin_V^e(c \otimes d)\big]+\\
        &\qquad + 2(\omega^{\op} \otimes \Id_{\pop} \otimes \Id_{\Poly})\big[\big((\nabla a)^{\ast} \cdot \nabla c\big) \otimes b^{\ast}d\big]+\\
        &\qquad + 2(\Id_{ \pop } \otimes \omega^{\op} \otimes \Id_{\Poly})\big[ca^{\ast} \otimes \big((\nabla b)^{\ast} \cdot \nabla d\big)\big]\;. 
    \end{equs} 
\end{lemma}

\begin{remark} 
    We note that in \eqref{eq:product_L2} $\nabla_i$ is viewed both as a map $\Poly \to \CP^{e}$ in the last line and $\pop \to (\pop)^e$ in the penultimate line.
 \end{remark}

\begin{proof}
    Both identities follow from \eqref{eq:productL}. For \eqref{eq:product_IdL}, observe that
    \begin{equs}
        (\Id_{\pop} \otimes \Lin_V)\big[(a \otimes b)^{\ast} \cdot (c \otimes d)\big] &= (\Id_{\pop} \otimes \Lin_V)\big[ca^* \otimes b^*d \big] =\\
        &=ca^{\ast} \otimes \Lin_V(b^{\ast}d) =  \\
        &=ca^* \otimes \left( (\Lin_V b)^* d + b^* \Lin_V d + 2(\omega^{\op} \otimes \bone)\left[ (\nabla b)^* \cdot (\nabla d) \right] \right) = \\
        &=\left(a \otimes \Lin_V b \right)^* \cdot (c \otimes d) + (a \otimes b)^* \cdot (c \otimes \Lin_V d) + \\
        & \qquad + 2( \bone_{ \pop } \otimes \omega^{\op} \otimes  \bone_{\Poly}) \left[ ca^* \otimes \left( (\nabla b)^* \cdot \nabla d \right) \right]\;,
    \end{equs}
    and apply \eqref{eq:productL} to $\Lin_V(b^{\ast}d)$, using Lemma~\ref{lem:generator_star_preserving}.

    Identity \eqref{eq:product_L2} follows by adding the $(\Lin_V^{\op}\otimes\Id)$ and $(\Id\otimes\Lin_V)$ contributions.
\end{proof}

We now introduce some additional notation along with the formulae for the noncommutative Hessians that appear in our analysis.

Given an algebra $\CA$, we write $\mathrm{M}_{n}(\CA)$ for the algebra of $n \times n$ matrices with elements in $\CA$. We write $\Tr \colon \mathrm{M}_{n}(\mathcal{A}) \mapsto \mathcal{A}$ for the algebra valued trace, that is $A = (a_{ij})_{i,j=1}^{n} \mapsto \sum_{i=1}^{n} a_{ii}$. We also write $\mathrm{M}_{n}(\CA) \ni A,B \mapsto AB \in \mathrm{M}_{n}(\CA)$ for the standard matrix product.

If $\mathcal{A}$ is a $*$-algebra, we define a corresponding conjugate-transpose involution $A = (a_{ij})_{i,j =1}^{n} \mapsto (a^{*}_{ji})_{i,j=1}^{n} = A^{*}$ on $\mathrm{M}_{n}(\CA)$.

\begin{definition}[Noncommutative Hessians]\label{def:nc_hessian}
    For $f \in \Poly$, define
    \begin{equs}
        \mathrm{Hess}^{\mathfrak{l}}(f) &= \big(\mathrm{Hess}_{ij}^{\mathfrak{l}}(f)\big)_{i,j=1}^{m}\in \mathrm{M}_{m}(\Poly\otimes\pop\otimes\Poly) \;, \\
    \mathrm{Hess}^{\mathfrak{r}}(f) &= \big(\mathrm{Hess}^{\mathfrak{r}}_{ij}(f) \big)_{i,j=1}^{m} \in \mathrm{M}_{m}(\pop\otimes\pop\otimes\Poly)\;,
    \end{equs}
    by setting, for $i,j \in [m]$
    \begin{equ}
        \mathrm{Hess}_{ij}^{\mathfrak{l}}(f) = (\nabla_{i} \otimes \Id)\nabla_{j}f\;, \qquad \mathrm{Hess}_{ij}^{\mathfrak{r}}(f) = (\Id \otimes \nabla_{i})\nabla_{j}f\;.
    \end{equ}
\end{definition}

Given $f, g \in \Poly$, we define $\mathcal{H}^{\mathfrak{l}}(f,g), \mathcal{H}^{\mathfrak{r}}(f,g), \mathcal{H}(f,g) \in \Poly^{\otimes 2}$ by setting
\begin{equs}
    \mathcal{H}^{\mathfrak{l}}(f,g) &= (\omega \otimes \Id_{\pop} \otimes \Id_{\Poly})\Tr\Big(\mathrm{Hess}^{\mathfrak{l}}(f)^{\ast}  \mathrm{Hess}^{\mathfrak{l}}(g)\Big)\;,\\
\mathcal{H}^{\mathfrak{r}}(f,g) &= (\Id_{\pop} \otimes \omega \otimes \Id_{\Poly})\Tr\Big(\mathrm{Hess}^{\mathfrak{r}}(f)^{\ast}  \mathrm{Hess}^{\mathfrak{r}}(g)\Big)\;,\\
\mathcal{H}(f,g) &= \mathcal{H}^{\mathfrak{l}}(f,g) + \mathcal{H}^{\mathfrak{r}}(f,g)\;.
\end{equs}
On the right-hand sides of the first two formulae above, we are using the matrix product. Note that $\CH^{\mfl}(f,f) \geqslant 0$ and $\CH^{\mfr}(f,f) \geqslant 0$.

\begin{lemma}\label{lem:hessian_product}
    For $f, g \in \Poly$,
\begin{equ}\label{eq:hessian_product_IdL} (\Id \otimes \Lin_V)\Big[(\nabla f)^{\ast} \cdot \nabla g\Big] = \big[(\Id \otimes \Lin_V)(\nabla f)^{\ast}\big] \cdot \nabla g + (\nabla f)^{\ast} \cdot \big[(\Id \otimes \Lin_V)\nabla g\big] + 2\mathcal{H}^{\mathfrak{r}}(f,g)\;,
\end{equ}
\begin{equs}[eq:hessian_product_L2]
    \Lin_V^e\Big[(\nabla f)^{\ast} \cdot \nabla g\Big] &= \big[\Lin_V^e(\nabla f)^{\ast}\big] \cdot \nabla g + (\nabla f)^{\ast} \cdot \big[\Lin_V^e\nabla g\big] + 2\mathcal{H}(f,g)\;.
\end{equs}
\end{lemma}

\begin{proof}
    We prove \eqref{eq:hessian_product_L2}; the identity \eqref{eq:hessian_product_IdL} follows by a simpler argument.

    Fix $j \in [m]$ and write $\nabla_{j}f = \sum_{k} a_{k} \otimes b_{k}$ and $\nabla_{j}g = \sum_{k'} c_{k'} \otimes d_{k'}$ with $a_k , c_{k'} \in \pop$ and $b_k, d_{k'} \in \Poly$. By \eqref{eq:product_L2},
    \begin{equs}
        \Lin_V^e\big[(\nabla_{j} f)^{\ast} \cdot (\nabla_{j} g)\big] &= \big[\Lin_V^e(\nabla_{j} f)^{\ast}\big] \cdot (\nabla_{j} g) + (\nabla_{j} f)^{\ast} \cdot \big[\Lin_V^e(\nabla_{j} g)\big]\\
    &\quad + 2\sum_{k,k'}\sum_{i=1}^{m}\bigg[(\omega^{\op} \otimes \Id_{\pop} \otimes \Id_{\Poly})\big[(\nabla_{i}a_{k})^{\ast} \cdot (\nabla_{i}c_{k'}) \otimes b_{k}^{\ast}d_{k'}\big]\\
    &\qquad\qquad + (\Id_{\pop} \otimes \omega^{\op} \otimes \Id_{\Poly})\big[c_{k'} a_{k}^{\ast} \otimes ((\nabla_{i}b_{k})^{\ast} \cdot \nabla_{i}d_{k'})\big]\bigg]\;.
    \end{equs}
    The result follows from observing that, for fixed $i \in [m]$,
    \begin{equs}
        \sum_{k,k'}\big((\nabla_{i} a_{k})^{\ast} \cdot (\nabla_{i} c_{k'})\big) \otimes b_{k}^{\ast}d_{k'} &= \Big((\nabla_{i} \otimes \Id)\nabla_{j}f\Big)^{\ast} \cdot \Big((\nabla_{i} \otimes \Id)\nabla_{j}g\Big)\;,\\
    \sum_{k,k'} c_{k'} a_{k}^{\ast} \otimes \big((\nabla_{i}b_{k})^{\ast} \cdot (\nabla_{i}d_{k'})\big) &= \Big((\Id \otimes \nabla_{i})\nabla_{j}f\Big)^{\ast} \cdot \Big((\Id \otimes \nabla_{i})\nabla_{j}g\Big)\;.
    \end{equs}
    Summing over $i$ and $j$ gives the Hessian terms $\mathcal{H}^{\mathfrak{l}}(f,g)$ and $\mathcal{H}^{\mathfrak{r}}(f,g)$.
\end{proof}

\subsection{Iterated Carr\'{e} du Champ}\label{ss:nc_gamma2}

\subsubsection{First Formulation: \TitleEquation{\Gamma_{2}}{Gamma\_2}}

We start by formally differentiating the expression underlying \eqref{eq:gradientestimate1}.
Setting $F(s) = P_{t-s}f$, and $\Psi_K(s) = e^{-2Ks} P_s\mbig[\Gamma(F(s))\mbig]$, we compute
\begin{equs}
    \Psi_{K}'(s) &= e^{-2Ks}\, P_{s}\Big[-2K\,\Gamma(F) + (\omega^{\op} \otimes \Lin_V)\big[(\nabla F)^{\ast} \cdot (\nabla F)\big] - \\
    &\qquad\qquad - (\omega^{\op} \otimes \Id)\big[(\nabla \Lin_V F)^{\ast} \cdot (\nabla F)\big] - (\omega^{\op} \otimes \Id)\big[(\nabla F)^{\ast} \cdot (\nabla \Lin_V F)\big]\Big]=\\
    &= 2e^{-2Ks}\, P_{s}\big[\Gamma_{2}(F) - K\,\Gamma(F)\big]\;, \label{eq:Psi_deriv}
\end{equs}
where the iterated carr\'{e} du champ is defined for $f,g \in \Poly$ by 
\begin{equs}[eq:nc_gamma2_def1]
    \Gamma_{2}(f,g) \eqdef & \frac{1}{2}\Big(\Lin_V\,\Gamma(f,g) - \Gamma(f, \Lin_V g) - \Gamma(\Lin_V f, g)\Big) = \\
    =& \frac{1}{2}\Big((\omega^{\op} \otimes \Lin_V)\big[(\nabla f)^{\ast} \cdot (\nabla g)\big] - (\omega^{\op} \otimes \Id)\big[(\nabla\Lin_V f)^{\ast} \cdot (\nabla g)\big] - \\
    &\quad - (\omega^{\op} \otimes \Id)\big[(\nabla f)^{\ast} \cdot (\nabla\Lin_V g)\big]\Big) =\\
    =& \frac{1}{2} (\omega^{\op} \otimes \Id) \Big(\big(\llbracket\Lin_V,\nabla\rrbracket_{\mfr} f\big)^{\ast} \cdot \nabla g + (\nabla f)^{\ast} \cdot \big(\llbracket\Lin_V ,\nabla\rrbracket_{\mfr} g\big) + 2 \CH^{\mfr}(f,g)\Big)\;.
\end{equs}
In the third equality, we used the product formula (Lemma~\ref{lem:hessian_product}) and introduced
\begin{equ}[eq:commutator1]
    \llbracket A, \nabla_{i}\rrbracket_{\mfr} \eqdef (\Id \otimes A)\nabla_{i} -\nabla_{i}A \; ,
    \quad
    \llbracket A, \nabla\rrbracket_{\mfr}f \eqdef \big(\llbracket A, \nabla_{i}\rrbracket_{\mfr} f\big)_{i=1}^{m} \; .
\end{equ}

To make the motivating calculation rigorous we need to extend $\Gamma_2$ to a sufficiently large $P_t$-invariant domain. 
By Remark~\ref{rem:DomainAss} together with Corollary~\ref{cor:ProdDomain} and the analogous factorwise argument based on \eqref{eq:product_IdL}, we see that the first term is well-defined on $D\bigl((-\Lin_V)^{3/2}\bigr)^2$. Similarly, by functional calculus, we see that on this domain, $\Lin_V f \in D\bigl((-\Lin_V)^{1/2}\bigr) = D(\CE)$. Therefore, we obtain the extension
\begin{equ}
    \Gamma_2 \colon D\bigl((-\Lin_V)^{3/2}\bigr) \times D\bigl((-\Lin_V)^{3/2}\bigr) \longrightarrow \CL^1(\omega)
\end{equ}
on a $P_s$-invariant domain. Note that $\Lin_V \Gamma(f,g)$ should be interpreted as $(\Lin_V)_1 \Gamma(f,g)$.

The rest of \eqref{eq:nc_gamma2_def1} could be justified analogously if one extended $\CH^\mfr$, however, this will not be necessary for us here. 

This motivates the following definition. 
\begin{definition}
    We say $V$ satisfies a $\Gamma_2$-criterion with constant $K \geqslant 0$ if for all $a \in D\bigl((-\Lin_V)^{3/2}\bigr)$, 
    \begin{equ}
        \Gamma_2(a) \geqslant K\Gamma(a)
    \end{equ}
    in $\CL^1(\omega)_+$. 
\end{definition}

\begin{remark}
    Since the map
    \begin{equ}
        a \longmapsto \Gamma_2(a)-K\Gamma(a)
    \end{equ}
    is continuous from $D((-\Lin_V)^{3/2})$, equipped with its graph norm, into $\CL^1(\omega)$, and $\CL^1(\omega)_+$ is closed, the graph-norm density of $\Poly$ shows that it is sufficient to verify the $\Gamma_2$-criterion for $a\in\Poly$. 
    
    For such $a$, the element $\Gamma_2(a)-K\Gamma(a)$ belongs to $\CM_{\mathrm{sa}}$, so by \eqref{eq:positivity} its positivity may equivalently be tested against
    all states on $\CM$.
\end{remark}

We then have the following lemma. 

\begin{lemma}\label{lem:gamma2_comparison}
    Suppose $V$ satisfies a $\Gamma_2$-condition with constant $K \geqslant 0$. Then the weak gradient estimate \eqref{eq:gradientestimate1} holds with the same constant $K$.
\end{lemma}
\begin{proof}
    If $f\in\Poly$, then $F(s)=P_{t-s}f$ belongs to $D((-\Lin_V)^{3/2})$ by the invariance of this domain. Hence the $\Gamma_2$-criterion and \eqref{eq:Psi_deriv} give $\Psi_K'(s)\geqslant0$ for every $s\in[0,t]$. Evaluating at the endpoints proves \eqref{eq:gradientestimate1} on $\Poly$, and graph-norm density and continuity extend it to $D(\overline\nabla)$.
\end{proof}

\subsubsection{Second Formulation: \TitleEquation{\tilde{\Gamma}_{2}}{Gamma\_2}}

Define the strong commutator and the lifted iterated carr\'{e} du champ by
\begin{equ}[eq:commutator2]
    \llbracket\Lin_V,\nabla_i\rrbracket \eqdef \Lin_V^e\nabla_i-\nabla_i\Lin_V\;,\qquad \llbracket\Lin_V,\nabla\rrbracket f \eqdef \bigl(\llbracket\Lin_V,\nabla_i\rrbracket f\bigr)_{i=1}^m\;,
\end{equ}
and
\begin{equs}[eq:nc_gamma2_tilde]
    \tilde\Gamma_2(f,g) &\eqdef \frac12\Big(\Lin_V^e\tilde\Gamma(f,g)-\tilde\Gamma(\Lin_Vf,g)-\tilde\Gamma(f,\Lin_Vg)\Big) =\\
&= \frac12\bigl(\llbracket\Lin_V,\nabla\rrbracket f\bigr)^*\cdot\nabla g + \frac12(\nabla f)^*\cdot\bigl(\llbracket\Lin_V,\nabla\rrbracket g\bigr) + \CH(f,g)\;.
\end{equs}
The second equality follows from Lemma~\ref{lem:hessian_product}; the factor $1/2$ in the definition and the factor $2$ in the Hessian product formula are both essential. Invariance implies
\begin{equ}[eq:gamma2_lift]
    \Gamma_2(f,g) = (\omega^{\op}\otimes\Id)\bigl(\tilde\Gamma_2(f,g)\bigr)\;.
\end{equ}
Moreover,
\begin{equ}[eq:Psitilde_deriv]
    \frac{\d}{\d s}\left[e^{-2Ks}P_s^e\tilde\Gamma(F(s))\right] = 2e^{-2Ks}P_s^e\bigl(\tilde\Gamma_2(F(s))-K\tilde\Gamma(F(s))\bigr)\;.
\end{equ}
The same domain considerations as in the preceding subsection continue to hold.

\begin{definition}[Lifted $\tilde\Gamma_2$ Criterion]
    We say that the lifted $\tilde\Gamma_2$-criterion holds with constant $K\in\RR$ if for all $a \in D((-\Lin_V)^{\frac{3}{2}})$
\begin{equ}
    \tilde\Gamma_2(a)\geqslant K\tilde\Gamma(a)
\end{equ}
    in $\CL^1(\omega^e)_+$.
\end{definition}

\begin{remark}
    It is again sufficient to check this condition merely on $\Poly$.
\end{remark}

\begin{lemma}\label{lem:gamma2tilde_comparison}
    If the completed lifted $\tilde\Gamma_2$-criterion holds with constant $K$, then \eqref{eq:gradientestimate2} holds with the same $K$.
\end{lemma}

\begin{proof}
    Equation~\eqref{eq:Psitilde_deriv} shows that the completed order-valued interpolation is increasing. Evaluating it at the endpoints gives \eqref{eq:gradientestimate2}.
\end{proof}

\subsection{Commutator Identities}\label{ss:commutator}

We now express on $\Poly$ the curvature in terms of the free Hessian of $V$. For comparison with the weak calculus, set
\begin{equ}[eq:weak_commutator_definition]
    \llbracket\Lin_V,\nabla_i\rrbracket_\omega \eqdef (\omega^{\op}\otimes\Lin_V)\nabla_i-(\omega^{\op}\otimes\Id)\nabla_i\Lin_V\;.
\end{equ}

\begin{lemma}\label{lem:commutator_identity}
    \begin{equ}[eq:commutator_identity]
        (\omega^{\op}\otimes\Id)\llbracket\Lin_V,\nabla_i\rrbracket = \llbracket\Lin_V,\nabla_i\rrbracket_\omega\;.
    \end{equ}
\end{lemma}
\begin{proof}
    By \eqref{eq:standing_invariance}, $\omega\Lin_V=0$, and hence $\omega^{\op}\Lin_V^{\op}=0$. Therefore,
    \begin{equs}
        (\omega^{\op}\otimes\Id)\llbracket\Lin_V,\nabla_i\rrbracket &= (\omega^{\op}\otimes\Id)\bigl((\Lin_V^{\op}\otimes\Id+\Id\otimes\Lin_V)\nabla_i-\nabla_i\Lin_V\bigr) =\\
        &= (\omega^{\op}\otimes\Lin_V)\nabla_i-(\omega^{\op}\otimes\Id)\nabla_i\Lin_V\;.
    \end{equs}
\end{proof}

\begin{lemma}\label{lem:Laplacian_commutator}
    For all $i \in [m]$,
\begin{equ}
    \llbracket\Delta_{\omega}, \nabla_{i}\rrbracket = 0\;.
\end{equ}
\end{lemma}

\begin{proof}
    By definition $\llbracket\Delta_\omega,\nabla_i\rrbracket=\Delta_\omega^e\nabla_i-\nabla_i\Delta_\omega$. Since $\nabla_i$ is the opposite-first-factor form of $\partial_i$, the claim is equivalent to
\begin{equ}[eq:Laplacian_commutator_reduced]
    \partial_{i}\Delta_{\omega} = (\Delta_{\omega}\otimes\Id + \Id\otimes\Delta_{\omega})\,\partial_{i}\;,
\end{equ}
in $\Poly\otimes\Poly$.

Both sides of \eqref{eq:Laplacian_commutator_reduced} are linear, so it suffices to check the identity on a monomial $\indet^{w}$, $w \in \Words$. From $\Delta = 2\sum_{j}(\partial_{j}\otimes\Id)\partial_{j}$ and $\bar{\eta}(x\otimes y\otimes z) = \omega(y)\,xz$,
\begin{equ}[eq:Delta_omega_monomial]
    \Delta_{\omega}\indet^{w} = \!\!\sum_{w = u\,c\,v\,c\,s}\!\!\omega(\indet^{v})\,\indet^{us}\;,
\end{equ}
the sum running over ordered pairs of positions of $w$ holding a common letter $c$, with $u,v,s$ the segments lying before, between, and after them. Applying $\partial_{i}$ and splitting the product $\indet^{us}$ at the chosen letter $i$ --- which lies in $u$ or in $s$, never in the contracted segment $v$ --- gives
\begin{equs}
    \partial_{i}\Delta_{\omega}\indet^{w} = \!\!\sum_{w = u\,c\,v\,c\,s}\!\!\omega(\indet^{v}) \Big[ \sum_{u = \alpha\,i\,u'}\indet^{\alpha}\otimes\indet^{u' s} + \sum_{s = s'\,i\,\beta}\indet^{u s'}\otimes\indet^{\beta}\Big]\;.
\end{equs}
On the other hand, writing $\partial_{i}\indet^{w} = \sum_{w = \xi\,i\,\eta}\indet^{\xi}\otimes\indet^{\eta}$ and applying \eqref{eq:Delta_omega_monomial} in each slot,
\begin{equs}
    (\Delta_{\omega}\otimes\Id)\,\partial_{i}\indet^{w} &= \!\!\sum_{w = \xi\,i\,\eta}\ \sum_{\xi = u\,c\,v\,c\,s}\!\!\omega(\indet^{v})\,\indet^{us}\otimes\indet^{\eta}\;,\\
(\Id\otimes\Delta_{\omega})\,\partial_{i}\indet^{w} &= \!\!\sum_{w = \xi\,i\,\eta}\ \sum_{\eta = u\,c\,v\,c\,s}\!\!\omega(\indet^{v})\,\indet^{\xi}\otimes\indet^{us}\;.
\end{equs}
Every term on either side is indexed by a choice of one occurrence of the letter $i$ together with an ordered pair of occurrences of a common letter $c$, the $i$ lying outside the block $c\,v\,c$. When that block lies to the right of $i$ the contraction commutes with the split, and the term matches the corresponding summand of $(\Id\otimes\Delta_{\omega})\partial_{i}\indet^{w}$ (the first group in $\partial_{i}\Delta_{\omega}\indet^{w}$); when it lies to the left it matches the corresponding summand of $(\Delta_{\omega}\otimes\Id)\partial_{i}\indet^{w}$ (the second group). This is a bijection between the index sets, so the two sides of \eqref{eq:Laplacian_commutator_reduced} coincide.
\end{proof}

For the drift, define $T_{j,V}\colon\Poly\to\Poly$ by
\begin{equ}
    T_{j,V}f \eqdef -(\partial_jf)\tw D_jV\;,
\end{equ}
so that $\Lin_V=2\Delta_\omega+\sum_jT_{j,V}$.

\begin{lemma}\label{lem:drift_commutator}
    For any $i, j \in [m]$,
    \begin{equ}
        \llbracket T_{j,V},\nabla_i\rrbracket = T_{ij,V}\nabla_j\;,
    \end{equ}
    where $T_{ij,V} \colon \CP^{e} \rightarrow \CP^{e}$ is defined by
\begin{equ}
    T_{ij,V}\xi \eqdef (\partial_iD_jV)\cdot\xi\;.
\end{equ}
\end{lemma}

\begin{proof}
    Write $\partial_jf=\sum_k f_k^{(1)}\otimes f_k^{(2)}$. Applying $\partial_i$ to $T_{j,V}f$ gives the two terms obtained by differentiating $f_k^{(1)}$ and $f_k^{(2)}$, together with
    \begin{equ}
        -\sum_k(\partial_iD_jV)\cdot\bigl((f_k^{(1)})^{\op}\otimes f_k^{(2)}\bigr) = -T_{ij,V}\partial_j f\;.
    \end{equ}
    The first two terms are exactly $T_{j,V}^e\partial_if$ because $(\partial_j\otimes\Id)\partial_i=(\Id\otimes\partial_i)\partial_j$. Hence $T_{j,V}^e\partial_i-\partial_i T_{j,V}=T_{ij,V}\partial_j$.
\end{proof}

\begin{proposition}\label{prop:commutator}
    For any $i \in [m]$,
    \begin{equ}
        \llbracket\Lin_V, \nabla_{i}\rrbracket = \sum_{j=1}^{m} T_{ij,V}\,\nabla_{j}\;, \qquad  \llbracket\Lin_V, \nabla_{i}\rrbracket_{\omega} = (\omega^{\op} \otimes \Id) \sum_{j=1}^{m} T_{ij,V}\,\nabla_{j}\;.
    \end{equ}
\end{proposition}
\begin{proof}
    The first equality follows from Lemmata~\ref{lem:Laplacian_commutator} and~\ref{lem:drift_commutator}. Applying Lemma~\ref{lem:commutator_identity} gives the second.
\end{proof}

View $(T_{ij,V})_{i,j}$ as the matrix $\vec T_V\in\mathrm M_m(\Poly^e)$ acting on $(\Poly^e)^m$ by
\begin{equ}
    \vec{T}_{V}\vec{f} = \bigg(\sum_{j=1}^{m}T_{ij,V} \cdot f_{j}\bigg)_{i=1}^{m}\;.
\end{equ}
This is the enveloping-algebra Hessian of $V$.

\begin{definition}[Full and Partial Free Convexity]
    Let $K \geqslant 0$. We call $V$ \emph{fully $K$-convex} if, for every $\xi\in(\Poly^e)^m$,
    \begin{equ}[eq:full_convexity]
        \frac12\left[\xi^*\cdot(\vec T_V\xi)+(\vec T_V\xi)^*\cdot\xi\right] - K\,\xi^*\cdot\xi \geqslant 0
    \end{equ}
    in the positive cone of $\CM^e$. We call $V$ \emph{partially $K$-convex} if applying $\omega^{\op}\otimes\Id$ to \eqref{eq:full_convexity} gives a positive element of $\CM$. 
\end{definition}

Full $K$-convexity implies partial $K$-convexity. At the algebraic level, this uses the same Hessian as $h$-convexity in \cite{DabrowskiGuionnetShlyakhtenko2022}.

\begin{proposition}\label{prop:convexity_implies_curvature}
    If $V$ is fully $K$-convex, then $\tilde\Gamma_2\geqslant K\tilde\Gamma$. If it is partially $K$-convex, then $\Gamma_2\geqslant K\Gamma$.
\end{proposition}
\begin{proof}
    Insert Proposition~\ref{prop:commutator} into \eqref{eq:nc_gamma2_tilde}. For $f=g \in \Poly$, the two commutator terms are the Hermitian part of $(\nabla f)^*\cdot\vec T_V\nabla f$, while $\CH(f,f)\geqslant0$. Equation~\eqref{eq:full_convexity} therefore gives $\tilde\Gamma_2(f)\geqslant K\tilde\Gamma(f)$. Applying $\omega^{\op}\otimes\Id$ and using \eqref{eq:gamma2_lift} proves the partial statement.

    Finally, the full $\tilde{\Gamma}_2$- and $\Gamma_2$-criteria follow by density of $\Poly$.
\end{proof}

We close this section with some examples of fully convex potentials $V$.

\begin{example}\label{ex:quartic_potentials}
    Let $V = (\sum_{i=1}^{m}\indet^{i})^{4}$ and $W = \sum_{i=1}^{m}(\indet^{i})^{4}$. We claim that both $V$ and $W$ are fully $0$-convex.

\begin{equ}
    DV = 4\Big(\sum_{i=1}^{m}\indet^{i}\Big)^{3}\begin{pmatrix}1\\\vdots\\1\end{pmatrix}\;, \quad DW = 4\begin{pmatrix}(\indet^{1})^{3}\\\vdots\\(\indet^{m})^{3}\end{pmatrix}\;.
\end{equ}
We then have
\begin{equs}
    \vec{T}_{V} &= 4\bigg[\Big(\sum_{i=1}^{m}\indet^{i}\Big)^{2} \otimes \bone + \Big(\sum_{i=1}^{m}\indet^{i}\Big) \otimes \Big(\sum_{i=1}^{m}\indet^{i}\Big) + \bone \otimes \Big(\sum_{i=1}^{m}\indet^{i}\Big)^{2}\bigg] \begin{pmatrix} 1 & \cdots & 1\\
\vdots & \ddots & \vdots\\
1 & \cdots & 1 \end{pmatrix}\;,\\
\vec{T}_{W} &= 4\begin{pmatrix} (\indet^{1})^{2} \otimes \bone + \indet^{1} \otimes \indet^{1} + \bone \otimes (\indet^{1})^{2} & & \\
& \ddots & \\
& & (\indet^{m})^{2} \otimes \bone + \indet^{m} \otimes \indet^{m} + \bone \otimes (\indet^{m})^{2} \end{pmatrix}\;.
\end{equs}
We now claim that any element $b \in \CP^{e} $ of the form $b \eqdef a^2 \otimes \bone + a \otimes a + \bone \otimes a^2$ is positive for any self-adjoint $a$. To see this, let $\psi$ be a state on $\CP^{e}$, then
\begin{equs}
    \psi(b) &= \psi(a^2 \otimes 1) + \psi(1 \otimes a^2)  + \psi(a \otimes a) \geqslant \\
& \geqslant \psi(a^2 \otimes 1) + \psi(1 \otimes a^2) -  |\psi(a \otimes a)| \geqslant \\
& \geqslant \psi(a^2 \otimes 1) + \psi(1 \otimes a^2) -  \sqrt{\psi(a^2 \otimes 1) \psi(1 \otimes a^2)} \geqslant \\
& \geqslant \psi(a^2 \otimes 1) + \psi(1 \otimes a^2)- \frac{1}{2}\left( \psi(a^2 \otimes 1) + \psi(1 \otimes a^2) \right) = \\
& = \frac{1}{2}\left( \psi(a^2 \otimes 1) + \psi(1 \otimes a^2) \right) \geqslant 0 \; .
\end{equs}
Here we used the Cauchy--Schwarz and arithmetic--geometric mean inequalities.

Finally, let $b_i = (\indet^{i})^2 \otimes \bone + \indet^{i} \otimes \indet^{i} + \bone \otimes (\indet^{i})^2$ and $b =\left(\sum_{i=1}^{m}\indet^{i}\right)^{2} \otimes \bone + \left(\sum_{i=1}^{m}\indet^{i}\right) \otimes \left(\sum_{i=1}^{m}\indet^{i}\right) + \bone \otimes \left(\sum_{i=1}^{m}\indet^{i}\right)^{2}$. We find that, for any state $\psi$ on $\CP^{e}$ and $\vec{f} \in (\CP^{e})^m$,
\begin{equs}
    \psi(\vec{f}^* \cdot \vec{T}_W \vec{f}) & = \sum_{i} \psi( f_i^* b_i f_i) \geqslant 0 \\
\psi(\vec{f}^* \cdot \vec{T}_V \vec{f}) & = \psi( \bar f^* b \bar f) \geqslant 0
\end{equs}
where $\bar f \eqdef \sum_i f_i$. This holds because $\psi(f_i^* \bigcdot f_i)$ and $\psi(\bar f^* \bigcdot \bar f)$ are positive functionals.

\end{example}

\section{Noncommutative Inequalities}

\subsection{Gradient Estimates}\label{ss:gradest_comm_nc}

\begin{proposition}[Weak Gradient Estimate]\label{prop:comm_gradient_estimate_nc}
The following are equivalent:
\begin{enumerate}
    \item $\Gamma_{2} \geqslant \kappa\,\Gamma$\;;
    \item For all $f \in \Poly$ and $t \geqslant 0$,
    \begin{equ}[eq:comm_gradient_estimate_NC]
        \Gamma(P_{t}f) \leqslant e^{-2\kappa t}\, P_{t}\Gamma(f)\;.
    \end{equ}
\end{enumerate}
\end{proposition}

\begin{proof}
\begin{itemize}
    \item[($\Rightarrow$)]
    For $0 \leqslant s \leqslant t$, let
    \begin{equ}[eq:Psi_comm_NC]
        \Psi(s) \eqdef e^{-2\kappa s}\, P_{s}\big(\Gamma(P_{t-s}f)\big)\;.
    \end{equ}
    Then
    \begin{equ}
        \Psi'(s) = e^{-2\kappa s}\, P_{s}\Big( -2\kappa\,\Gamma(P_{t-s}f) + 2\,\Gamma_{2}(P_{t-s}f) \Big) \geqslant 0 \; . 
    \end{equ}
    Thus, $\Psi(0) \leqslant \Psi(t)$, which gives \eqref{eq:comm_gradient_estimate_NC}.
    \item[($\Leftarrow$)]
    Dividing 
    \begin{equ}
        0 \leqslant e^{-2\kappa t}P_{t}\Gamma(f) - \Gamma(P_{t}f) = \Psi(t)-\Psi(0)
    \end{equ}
    by $2t$ and taking the limit $t \downarrow 0$ yields
    \begin{equ}
        \Gamma_2 (f) - \kappa \Gamma(f) \geqslant 0 \;. 
    \end{equ}
\end{itemize}
\end{proof}

Analogously one proves also the equivalence for the strong gradient estimate. 
\begin{proposition}[Strong Gradient Estimate]\label{prop:comm_gradient_estimate_nc_str}
The following are equivalent:
\begin{enumerate}
    \item $\tilde\Gamma_{2} \geqslant \kappa\,\tilde\Gamma$\;;
    \item For all $f \in \Poly$ and $t \geqslant 0$,
    \begin{equ}[eq:comm_gradient_estimate_NC_str]
        \tilde\Gamma(P_{t}f) \leqslant e^{-2\kappa t}\, P^e_{t}\tilde\Gamma(f)\;.
    \end{equ}
\end{enumerate}
\end{proposition}

\begin{corollary}[Poincar\'{e} Inequality from \eqref{eq:comm_gradient_estimate_NC}]\label{cor:PI_from_CD_nc}
If $\,\Gamma_{2} \geqslant \kappa\Gamma$ with $\kappa > 0$, then
\begin{equ}
    P_{t}(f^{2}) - (P_{t}f)^{2} \leqslant \frac{1 - e^{-2\kappa t}}{\kappa}\, P_{t}\Gamma(f)\;.
\end{equ}
\end{corollary}
\begin{proof}
    As in \eqref{eq:PhiGamma}, let $\Phi(s) \eqdef P_s \left( |P_{t-s} f|^2 \right) $ for $s \in [0,t]$. From \eqref{eq:productL}, we see that 
    \begin{equs}
        \Phi'(s) &= P_s \left( \Lin_V \left( |P_{t-s}f|^2 \right)\right)  - P_s \left( \left( \Lin_V P_{t-s} f \right)^* P_{t-s} f +(P_{t-s} f )^*  \Lin_V P_{t-s} f \right) = \\
        & = 2 P_s \Gamma( P_{t-s}f ) \; . 
    \end{equs}
    and thus, by a computation analogous to \eqref{eq:Psi_deriv}, $\Phi''(s) = 4P_{s}\Gamma_{2}(P_{t-s}f)$.
    The condition $\Gamma_{2} \geqslant \kappa\,\Gamma$ implies $\Phi''(s) \geqslant
    2\kappa\,\Phi'(s)$ as $P_s$ is positivity preserving.
    
    Therefore, by Gr\"{o}nwall's inequality for the final condition at $s = t$ (the proof of which holds verbatim as the control $2\kappa$ is a scalar), we find
    \begin{equ}
        \Phi'(s) \leqslant e^{2\kappa(s-t)} \Phi'(t) = 2e^{2\kappa(s-t)} P_t \Gamma(f) \;.
    \end{equ}
    Since $\Phi(t) - \Phi(0) = P_{t}(|f|^{2}) - |P_{t}f|^{2}$, integrating yields
    \begin{equ}
        P_{t}(|f|^{2}) - |P_{t}f|^{2} = \int_{0}^{t}\Phi'(s)\,\d s \leqslant \frac{1 - e^{-2\kappa t}}{\kappa}\, P_{t}\Gamma(f)\;. 
    \end{equ}
\end{proof}

\subsection{Poincar\'{e} Inequalities}\label{ss:NPI}

\begin{definition}[Poincar\'{e} Inequality]
    We say that $(P_{t}, \omega)$ satisfies a \emph{noncommutative Poincar\'{e} inequality} with constant $C > 0$ if, for all $f \in \Poly$,
    \begin{equ}
        \Var(f) \eqdef \omega(f^* f)- \big|\omega(f)\big|^2 \leqslant C\,\CE(f) = C \omega(\Gamma(f))\;.
    \end{equ}
\end{definition}

\begin{lemma}[Voiculescu's Universal Free Poincar\'{e} Estimate]
\label{lem:ConstDeriv}
    Set
    \begin{equ}
        R_{\omega}^{2} \eqdef \sum_{i=1}^{m}\|Y_{0}^{i}\|^{2} =\sum_{i=1}^{m}
        \|\digamma_{0}(\indet^{i})\|^{2}\;.
    \end{equ}
    Then, for every $f \in \Poly$,
    \begin{equ}[eq:universal_free_PI]
        \Var(f) \leqslant 2 R_{\omega}^{2}\, \omega\bigl(\Gamma(f)\bigr)\;.
    \end{equ}
    In particular,
    \begin{equ}
        \omega\bigl(\Gamma(f)\bigr)=0  \quad\Longleftrightarrow\quad f\in\C\bone \;.
    \end{equ}
\end{lemma}

\begin{proof}
    Define the linear map $J\colon\Poly\to\Poly^{e}$ by
    \begin{equ}
        Jf \eqdef f^{\op}\otimes\bone-\bop\otimes f \;.
    \end{equ}
    Expanding its squared norm and using traciality, we obtain
    \begin{equs}[eq:J_variance]
        \|Jf\|_{2,\omega\otimes\omega}^{2} &= \omega(f f^{\ast}) + \omega(f^{\ast}f) - \omega(f^{\ast})\omega(f) - \omega(f)\omega(f^{\ast}) = \\
        &= 2\left( \omega(f^{\ast}f)-|\omega(f)|^{2} \right) \;.
    \end{equs}
    We claim that
    \begin{equ}[eq:J_telescoping]
        J f =
        \sum_{i=1}^{m} J(\indet^i) \nabla_i f
        \;.
    \end{equ}

    By linearity, it suffices to prove this for a monomial $f = \indet^{i_1}\cdots\indet^{i_d}$. For $r\in\{0,\dots,d\}$, set $P_r \eqdef \indet^{i_1}\cdots\indet^{i_r} $,  $ S_r \eqdef \indet^{i_{r+1}}\cdots\indet^{i_d}$, and $P_0=S_d=\bone$. The contribution of the $r^{\text{th}}$ letter to
    $\nabla_{i_r}f$ is $P_{r-1}^{\op}\otimes S_r$. For these we have
    \begin{equs}
        J(\indet^i)  \left(P_{r-1}^{\op}\otimes S_r\right) &= P_r^{\op}\otimes S_r -  P_{r-1}^{\op}\otimes\indet^{i_r}S_r = \\
        &= P_r^{\op}\otimes S_r - P_{r-1}^{\op}\otimes S_{r-1} \;.
    \end{equs}
    Summing over all positions in $f$ yields
    \begin{equs}
        \sum_{i=1}^{m} J(\indet^i) \nabla_i f &= \sum_{r=1}^{d}\left( P_r^{\op}\otimes S_r - P_{r-1}^{\op}\otimes S_{r-1}\right) = \\
        &= P_d^{\op}\otimes S_d - P_0^{\op}\otimes S_0 = f^{\op}\otimes\bone-\bop\otimes f = J f\;.
    \end{equs}
    This proves \eqref{eq:J_telescoping}.

    After evaluation at $Y_0$, the element $J \indet^i$ becomes $ (Y_0^i)^{\op}\otimes\bone  - \bop\otimes Y_0^i$. Hence
    \begin{equ}
        \|J \indet^i\| \leqslant 2\|Y_0^i\| \;.
    \end{equ}
    By H\"older's inequality, we thus obtain
    \begin{equs}
        \|Jf\|_{\CL^2(\omega\otimes\omega)} &\leqslant \sum_{i=1}^{m} \|J(\indet^i)\nabla_i f\|_{\CL^2(\omega\otimes\omega)}\leqslant \sum_{i=1}^{m} \|J(\indet^i)\|\|\nabla_i f\|_{\CL^2(\omega\otimes\omega)} \leqslant \\
        &\leqslant 2\sum_{i=1}^{m} \|Y_0^i\| \|\nabla_i f\|_{\CL^2(\omega\otimes\omega)} \leqslant 2 R_{\omega} \left( \sum_{i=1}^{m} \|\nabla_i f\|_{\CL^2(\omega\otimes\omega)}^2\right)^{\frac{1}{2}} \;.
    \end{equs}
    From \eqref{eq:gamma_lift},
    \begin{equ}
        \sum_{i=1}^{m} \|\nabla_i f\|_{\CL^2(\omega\otimes\omega)}^2 = (\omega\otimes\omega)\bigl(\widetilde{\Gamma}(f)\bigr) =   \omega\bigl(\Gamma(f)\bigr) \;.
    \end{equ}
    Squaring the preceding estimate and using
    \eqref{eq:J_variance}, we find
    \begin{equ}
        2\left(
            \omega(f^{\ast}f)-|\omega(f)|^{2}
        \right)
        \leqslant
        4R_{\omega}^{2}\,
        \omega\bigl(\Gamma(f)\bigr)
        \;,
    \end{equ}
    which proves \eqref{eq:universal_free_PI}.

    If $\omega(\Gamma(f))=0$, the estimate gives
    \begin{equ}
        \omega\left(\bigl(f-\omega(f)\bone\bigr)^{\ast}\bigl(f-\omega(f)\bone\bigr)\right) = 0 \;.
    \end{equ}
    Faithfulness of $\omega$ therefore implies
    $f=\omega(f)\bone$. Conversely, every free difference quotient
    annihilates constants.
\end{proof}

\begin{remark}[Sharper Row-Operator Constant]
    Let
    \begin{equ}
        T(\xi_1,\dots,\xi_m)
        \eqdef
        \sum_{i=1}^{m} J(\indet^i)\cdot\xi_i
        \;.
    \end{equ}
    Since each $J(\indet^i)$ is self-adjoint, we have 
    \begin{equ}
        TT^* = \sum_{i = 1}^m J(\indet^i)^2 
    \end{equ}
    and therefore
    \begin{equ}
        \|T\|^{2}
        =
        \left\|
            \sum_{i=1}^{m} J(\indet^i)^{2}
        \right\|
        \;.
    \end{equ}
    Thus the constant $2R_{\omega}^{2}$ in
    \eqref{eq:universal_free_PI} can be replaced by
    \begin{equ}
        C_{\omega}^{\mathrm{row}} \eqdef \frac{1}{2} \left\|\sum_{i=1}^{m}J(\indet^i)^{2} \right\|_{\infty}  \leqslant 2\left\| \sum_{i=1}^{m}(Y_0^i)^2 \right\|_{\infty}\leqslant 2R_{\omega}^{2} \;.
    \end{equ}
\end{remark}

\begin{corollary}[Kernel of \TitleEquation{\overline{\nabla}}{Free Difference quotient}]
\label{cor:closed_free_difference_quotient}
    For all $f\in D(\overline{\nabla})$,
    \begin{equ}[eq:closed_universal_free_PI]
        \|f-\omega(f)\bone\|_{L^{2}(\omega)}^{2} \leqslant 2R_{\omega}^{2}  \|\overline{\nabla}f\|_{\CL^{2}(\omega^e)^{m}}^{2} \;.
    \end{equ}
    Consequently,
    \begin{equ}
        \ker\overline{\nabla} = \C\bone \;.
    \end{equ}
\end{corollary} 
\begin{proof}
    Let $f_n\in\Poly$ converge to $f$ in the graph norm of
    $\overline{\nabla}$. Then $\omega(f_n)\to\omega(f)$, and passing
    to the limit in \eqref{eq:universal_free_PI} proves
    \eqref{eq:closed_universal_free_PI}. The assertion concerning the
    kernel follows immediately.
\end{proof}

\begin{definition}[Noncommutative Ergodicity]
    We say that a pair $(P_t, \omega)$ is ergodic if for all $f \in \CL^2(\omega)$
    \begin{equ}
        \lim_{t \to \infty} P_t f  = \omega(f) \bone \quad\text{in }\CL^2(\omega)\;.
    \end{equ}
\end{definition}

\begin{proposition}
    Whenever $\omega$ is stationary w.r.t.\ $P_t$ and satisfies \eqref{eq:comm_gradient_estimate_NC} with $\kappa > 0$, the pair is ergodic.
\end{proposition}
\begin{proof}
    Let $f \in D(\overline{\nabla})$. For $f = \lambda \bone$ ergodicity is trivial, thus we assume that $f$ is not a multiple of the identity. Then $f' \eqdef f - \omega(f)\bone$ is orthogonal to the multiples of the identity and, by Corollary~\ref{cor:closed_free_difference_quotient}, $\Gamma(f')\neq0$. It follows by Proposition~\ref{prop:comm_gradient_estimate_nc} that
    \begin{equ}
        \lim_{t \to \infty} \omega\otimes\omega\left( \tilde\Gamma(P_t f' ) \right) = \lim_{t \to \infty} \omega\left( \Gamma(P_t f' ) \right) \leqslant \lim_{t \to \infty} e^{-2 \kappa t } \omega\left( \Gamma(f') \right) = \lim_{t \to \infty} e^{-2 \kappa t } \CE(f) = 0 \; . 
    \end{equ}
    Since $\omega(P_tf')=0$, the closed universal estimate \eqref{eq:closed_universal_free_PI} gives
    \begin{equ}
        \|P_tf'\|_{\CL^2(\omega)}^2\leqslant 2R_\omega^2\,\omega\bigl(\Gamma(P_tf')\bigr)\xrightarrow{t\to\infty}0 \;.
    \end{equ}
    Thus,
    \begin{equs}
        \lim_{t\to \infty} \omega\left( |P_t f|^2 \right) &= \lim_{t \to \infty} \omega \left( \left( \omega(f) \bone + P_t f' \right)^*\left( \omega(f) \bone + P_t f' \right)  \right) = \\
        & = |\omega(f)|^2 + \lim_{t \to \infty}  2\Re \Big( \overline{\omega(f)} \underbrace{\omega(P_t f')}_{ = \omega(f') = 0} \Big)  + \lim_{t \to \infty} \omega\left( \left| P_t f' \right|^2 \right) = |\omega(f)|^2 \; . 
    \end{equs}
    By density this extends to all $f \in \CL^2(\omega)$.
\end{proof}

\begin{proposition}\label{prop:CD_implies_PI_nc}
    Suppose that the semigroup $(P_t)_t$ satisfies the weak gradient estimate with constant $K > 0$ and is ergodic.
    Then $P_{t}$ satisfies the Poincar\'{e} inequality with constant $1/K$.
\end{proposition}

\begin{proof}
    For $f \in D(\Lin_V)$, let $\alpha(s) \eqdef P_{s}\left((P_{t-s}f)^*(P_{t-s}f)\right) \eqdef P_s \left( f(s)^* f(s) \right)$.
    Then, by the product formula and the weak gradient estimate,
    \begin{equs}
        \alpha'(s) &= P_s \left( \Lin_V \left( f(s)^* f(s) \right) - \left(\Lin_V f(s)\right)^* f(s) - f(s)^* \Lin_V f(s) \right) =\\
        &=2P_s\Gamma(P_{t-s}f) \leqslant 2e^{-2K(t-s)}P_sP_{t-s}\Gamma(f) = 2e^{-2K(t-s)}P_t\Gamma(f)
    \end{equs}
    for all $s \in [0,t]$.
    Applying $\omega$ to both sides, and integrating we obtain 
    \begin{equ}
        \omega\left( |f|^2 \right) - \omega\left(  \left|P_t f\right|^2 \right) = \omega\left( P_t |f|^2 \right) - \omega\left(  \left|P_t f\right|^2 \right) \leqslant \frac{1-e^{-2Kt}}{K}\,\omega(P_t\Gamma(f)) = \frac{1-e^{-2Kt}}{K}\,\omega(\Gamma(f))\;.
    \end{equ}
    Here we used the invariance of $\omega$. The second term on the left-hand side converges to $ \omega( |\omega(f) \bone |^2 ) = |\omega(f)|^2 $ by ergodicity as $t \uparrow \infty$.
    Thus, taking this limit gives the result on the dense set $D(\Lin_V)$. By form-core approximation, it extends to all $f \in D(\CE)$.
\end{proof}

\subsection{Logarithmic Sobolev Inequalities}\label{ss:LSI_NC}

The Poincar\'{e} inequality obtained in the preceding subsection controls the quadratic entropy. We now turn to the logarithmic entropy and prove a modified logarithmic Sobolev inequality. The argument has two main ingredients. First, we pass from the completed order-valued gradient estimate to a logarithmic-mean gradient estimate. We then use a variational characterisation of Fisher information to obtain exponential Fisher-information decay and integrate the entropy-dissipation identity.

We retain the standing symmetry and core Assumption~\ref{ass:standing_symmetry}, the closed energy structure of Proposition~\ref{prop:DirichletForm}, and the notation $P_t=e^{t\Lin_V}$. 

We make indirect use of the preceding subsection, by assuming that the closed Poincar\'{e} estimate
\begin{equ}[eq:closed_PI_for_MLSI]
    \|a-\omega(a)\bone\|_{\CL^2(\omega)}^2 \leqslant C_{\mathrm P}\, \CE(a) \; ,
\end{equ}
holds for some $C_{\mathrm P}<\infty$ for all $a \in D(\CE)$. In particular, one may always take $C_{\mathrm P}=C_{\omega}^{\mathrm{row}}$ and, when the weak gradient estimate holds with $K>0$, one may take $C_{\mathrm P}=C_{\omega}^{\mathrm{row}}\wedge K^{-1}$.

We start by proving a simple statement that allows us to take the logarithm of elements of $D(\overline{\nabla})$ and a corresponding chain rule. 
\begin{lemma}[Logarithmic Chain Rule]\label{lem:logarithmic_chain_rule}
    Let $\rho\in D(\overline{\nabla})$ satisfy
    \begin{equ}
        \eps\bone\leqslant\rho\leqslant R\bone
    \end{equ}
    for some $\eps,R>0$. Then $\log(\rho)\in D(\overline{\nabla})$ and
    \begin{equ}[eq:logarithmic_chain_rule]
        \overline{\nabla}\log(\rho)=\left(\overline{\nabla}\rho\right)\log^{[1]}(\rho^{\op}\otimes\bone,\bone\otimes\rho) \;,
    \end{equ}
    where
    \begin{equ}
        \logd (x,y)\eqdef\begin{cases}\frac{\log x-\log y}{x-y} \;, & \text{if }x\neq y \;,\\
        \frac{1}{x} \;, & \text{if }x=y \;.
        \end{cases}
    \end{equ}
\end{lemma}

\begin{proof}
    Set $q\eqdef 1-\frac{\eps}{R}<1$ and 
    \begin{equ}
        x\eqdef\bone-\frac{\rho}{R} \;.
    \end{equ}
    Then $0\leqslant x\leqslant q\bone$ and
    \begin{equ}
        \log(\rho)=(\log R)\bone-\sum_{n=1}^\infty\frac{x^n}{n} \;.
    \end{equ}
    By the Leibniz rule,
    \begin{equ}
        \overline{\nabla}(x^n)=\sum_{k=0}^{n-1}x^k(\overline{\nabla}x)x^{n-1-k} = \sum_{k=0}^{n-1}(\overline{\nabla}x) x^k \otimes x^{n-1-k} \;,
    \end{equ}
    and hence
    \begin{equ}
        \frac{1}{n}\left\|\overline{\nabla}(x^n)\right\|_{\CL^2(\omega^e)^m} \leqslant q^{n-1}\left\|\overline{\nabla}x\right\|_{\CL^2(\omega^e)^m} \;.
    \end{equ}
    Thus the logarithmic series and its termwise gradient converge in their respective $\CL^2$-spaces. Closedness of $\overline{\nabla}$ shows that $\log(\rho)\in D(\overline{\nabla})$ and permits termwise differentiation. Since $\rho^{\op}\otimes\bone$ and $\bone\otimes\rho$ commute, summing the resulting bivariate power series gives \eqref{eq:logarithmic_chain_rule}.
\end{proof}

To compensate for the $\logd$ factor produced by differentiating $\log$, we define for $\rho\in\CM_+$ 
\begin{equ}[eq:logarithmic_mean_tensor]
    \Theta_\rho \eqdef \int_0^1 (\rho^{\op})^s \otimes \rho^{1-s} \,\d s \in (\CM^e)_+\;.
\end{equ}
If $\rho\geqslant\eps\bone$ for some $\eps>0$, joint functional calculus for the commuting operators gives
\begin{equ}
    \Theta_\rho = \left(\logd(\rho^{\op}\otimes\bone,\bone\otimes\rho)\right)^{-1}\;.
\end{equ}

\begin{proposition}[Strong Gradient Estimate Implies the Logarithmic-Mean Estimate]\label{prop:tensor-bridge} 
    Recall $\CM^e$, $\omega^e$, and $P_t^e$ from \eqref{eq:completed_enveloping_algebra} and \eqref{eq:tensor_semigroup}. Suppose that for some $K\in\R$, the estimate
    \begin{equ}[eq:completed-strong]
        \tilde\Gamma(P_t a) \leqslant e^{-2Kt}P_t^e\bigl(\tilde\Gamma(a)\bigr) \;,
    \end{equ}
    holds in $\CL^1(\omega^e)_+$ for every $t \geqslant 0$ and $a \in D(\CE)$. For $\rho\in\CM_+$, define
    \begin{equ}[eq:weighted_gradient_norm]
        \|\overline{\nabla}a\|_\rho^2 \eqdef \omega^e\!\left(\Theta_\rho\tilde\Gamma(a)\right) = \Braket{\overline{\nabla}a, (\overline{\nabla}a)\Theta_\rho}_{\CL^2(\omega^e)^m} \;.
    \end{equ}
    Then
    \begin{equ}[eq:lmge]
        \|\overline{\nabla}P_t a\|_\rho^2 \leqslant e^{-2Kt} \|\overline{\nabla}a\|_{P_t\rho}^2 \;.
    \end{equ}
\end{proposition}

\begin{proof}
    
    We start by proving the comparison
    \begin{equ}[eq:theta-comparison]
        P_t^e(\Theta_\rho) \leqslant \Theta_{P_t\rho} \;.
    \end{equ}
    First suppose that $\rho\geqslant\eps\bone$ for some $\eps>0$. For $0<s<1$, the function $r\mapsto r^s$ is operator concave. The Choi--Davis--Jensen inequality for the unital, completely positive map $P_t$ therefore gives
    \begin{equ}[eq:power-jensen]
        P_t(\rho^s) \leqslant (P_t\rho)^s \;, \qquad P_t(\rho^{1-s}) \leqslant (P_t\rho)^{1-s} \;.
    \end{equ}
    See \cite[Theorem~1.20, Corollary~1.22(i)]{PecaricFurutaMicicSeo2013}.
    Combining these inequalities in the spatial tensor product yields
    \begin{equs}
        {}& \bigl((P_t\rho)^s\bigr)^{\op}\otimes(P_t\rho)^{1-s} - \bigl(P_t(\rho^s)\bigr)^{\op}\otimes P_t(\rho^{1-s}) =\\
        {}& \hspace*{4cm}\quad= \bigl((P_t\rho)^s-P_t(\rho^s)\bigr)^{\op} \otimes(P_t\rho)^{1-s} +\\
        {}&\hspace*{4cm}\quad\qquad+ \bigl(P_t(\rho^s)\bigr)^{\op} \otimes \bigl((P_t\rho)^{1-s}-P_t(\rho^{1-s})\bigr) \geqslant 0 \;.
    \end{equs}
    Since $P_t^e$ is bounded and \eqref{eq:logarithmic_mean_tensor} is a norm-Bochner integral for invertible $\rho$, integration gives
    \begin{equs}
        P_t^e(\Theta_\rho) &= \int_0^1 \bigl(P_t(\rho^s)\bigr)^{\op} \otimes P_t(\rho^{1-s}) \,\d s \leqslant\\
    &\leqslant \int_0^1 \bigl((P_t\rho)^s\bigr)^{\op} \otimes(P_t\rho)^{1-s} \,\d s = \Theta_{P_t\rho} \;.
    \end{equs}
    If $0\leqslant B\leqslant C$ in $\CL^1(\omega^e)_+$ and $A \in \CM^e_+$, traciality gives
    \begin{equ}
        \omega^e\!\left(A(C-B)\right) = \omega^e\!\left(A^{1/2}(C-B)A^{1/2}\right) \geqslant 0 \;.
    \end{equ}
    Hence \eqref{eq:completed-strong}, \eqref{eq:tensor_semigroup_L1_symmetry} and \eqref{eq:theta-comparison} imply
    \begin{equs}
        \|\overline{\nabla}P_t a\|_\rho^2 &= \omega^e\!\left(\Theta_\rho\tilde\Gamma(P_t a)\right) \leqslant\\
    &\leqslant e^{-2Kt} \omega^e\!\left(\Theta_\rho P_t^e\bigl(\tilde\Gamma(a)\bigr)\right) =\\
    &= e^{-2Kt} \omega^e\!\left(P_t^e(\Theta_\rho)\tilde\Gamma(a)\right) \leqslant\\
    &\leqslant e^{-2Kt} \omega^e\!\left(\Theta_{P_t\rho}\tilde\Gamma(a)\right) = e^{-2Kt} \|\overline{\nabla}a\|_{P_t\rho}^2 \;.
    \end{equs}

    It remains to remove the invertibility assumption. For $\rho\in\CM_+$, set $\rho_\eps\eqdef \rho+\eps\bone$. If $0<\eps<1$, functional calculus gives, for all $s \in (0,1)$,
    \begin{equ}
        \left\|(\rho+\eps\bone)^s-\rho^s\right\| \leqslant \eps^s\;.
    \end{equ}
    The identity $A_\eps\otimes B_\eps-A\otimes B =(A_\eps-A)\otimes B_\eps+A\otimes(B_\eps-B)$ therefore yields
    \begin{equ}
        \left\|\Theta_{\rho_\eps}-\Theta_\rho\right\|\leqslant (1+\|\rho\|)\int_0^1 \left(\eps^s+\eps^{1-s}\right) \,\d s= 2(1+\|\rho\|)\frac{(1-\eps)}{|\log\eps|} \xrightarrow{\eps\downarrow0} 0 \;.
    \end{equ}
    Since $P_t\rho_\eps=P_t\rho+\eps\bone$, the same argument also gives
    \begin{equ}
        \Theta_{P_t\rho_\eps} \longrightarrow \Theta_{P_t\rho} 
    \end{equ}
    in operator norm. Applying the established inequality to $\rho_\eps$ and passing to the limit proves \eqref{eq:lmge} for every bounded $\rho\geqslant0$.
\end{proof}

For $f\in\CL^1(\omega)_+$, define its extended-valued entropy by
\begin{equ}[eq:homogeneous_entropy]
    \Ent_\omega(f) \eqdef \omega(f\log f) - \omega(f)\log\omega(f) \;,
\end{equ}
where $0\log0\eqdef 0$. A density is an element $\rho\in\CL^1(\omega)_+$ with $\omega(\rho)=1$, in which case
\begin{equ}
    \Ent_\omega(\rho) = \omega(\rho\log(\rho)) \;.
\end{equ}

For bounded $\rho\in D(\overline{\nabla})_+$, we define the relaxed Fisher information by
\begin{equ}[eq:relaxed_fisher_bounded]
    \CI_\omega(\rho) \eqdef \lim_{\eps\downarrow0} \CE\left(\rho,\log(\rho+\eps\bone)\right) \in [0,\infty] \;.
\end{equ}
The divided-difference kernels increase as $\eps \downarrow 0$, so the limit exists. For bounded $\rho\in\CM_+\setminus D(\overline{\nabla})$, set $\CI_\omega(\rho)=\infty$. For arbitrary $\rho\in\CL^1(\omega)_+$, set
\begin{equ}[eq:relaxed_fisher]
    \CI_\omega(\rho) \eqdef \liminf_{R \to \infty} \CI_\omega(\rho\wedge R) \;,
\end{equ}
where the quantity on the right is the bounded Fisher information from \eqref{eq:relaxed_fisher_bounded}. By \eqref{eq:DirProp}, if $\rho\in D(\overline{\nabla})_+$, then $\rho\wedge R\in D(\overline{\nabla})$ for every $R\geqslant0$. The bounded and extended definitions agree when $\rho$ is bounded, because $\rho\wedge R=\rho$ for $R\geqslant\|\rho\|$. Moreover, for $c>0$,
\begin{equ}[eq:fisher_homogeneity]
    \CI_\omega(c\rho)=c\CI_\omega(\rho) \;.
\end{equ}
Indeed, the bounded definition is homogeneous after rescaling $\eps$, and $(c\rho)\wedge R=c(\rho\wedge(R/c))$; the use of a continuous truncation parameter in \eqref{eq:relaxed_fisher} is important here.

This is a slight modification of the standard relaxation of the Fisher information, cf.\ \cite[Section~5]{Wirth2018}, which we employ to bypass proving monotonicity of $s \mapsto \CI_{\omega}(\rho \wedge s)$ and is sufficient for our application. 

\begin{definition}[Modified Logarithmic Sobolev Inequality]
    We say that $(P_t)_{t\geqslant0}$ satisfies the modified logarithmic Sobolev inequality (MLSI) with constant $\alpha>0$ if
    \begin{equ}[eq:definition_MLSI]
        \alpha\Ent_\omega(\rho) \leqslant \CI_\omega(\rho) \;,
    \end{equ}
    for every density $\rho$.    
\end{definition}

\begin{theorem}[Logarithmic-Mean Gradient Estimate Implies MLSI]\label{thm:lmge_implies_mlsi} 
    Suppose that $K>0$ and that \eqref{eq:lmge} holds for every $a\in D(\overline{\nabla})$, every bounded $\rho\in\CM_+$ and every $t\geqslant0$. Then $(P_t)_{t\geqslant0}$ satisfies the modified logarithmic Sobolev inequality with constant $2K$; equivalently, every density satisfies
\begin{equ}[eq:MLSI_nc]
    \Ent_\omega(\rho) \leqslant \frac{1}{2K}\CI_\omega(\rho) \;.
\end{equ}
\end{theorem}

\begin{proof}
    We first consider a density $\rho\in D(\overline\nabla)$ such that, for some $\eps,R>0$, $\eps\bone \leqslant \rho \leqslant R\bone$. Lemma~\ref{lem:logarithmic_chain_rule} then applies.
    Consequently,
    \begin{equ}[eq:log_mean_chain_rule]
        (\overline{\nabla}\log(\rho))\Theta_\rho = \overline{\nabla}\rho \;.
    \end{equ}
    For vectors $a,b \in \CL^2(\omega^e)^m$, write
    \begin{equ}
        \Braket{a,b}_\rho \eqdef \Braket{a,b\Theta_\rho}_{\CL^2(\omega^e)^m} \;.
    \end{equ}
    Equation~\eqref{eq:log_mean_chain_rule} and traciality of $\omega^e$ yield
    \begin{equ}[eq:energy_log_mean_duality]
        \CE(\rho,a) = \Braket{\overline{\nabla}\rho, \overline{\nabla}a}_{\CL^2(\omega^e)^m} = \Braket{\overline{\nabla}\log(\rho), \overline{\nabla}a}_\rho \;.
    \end{equ}
    In particular,
    \begin{equ}
        \CI_\omega(\rho) = \CE(\rho,\log(\rho)) = \|\overline{\nabla}\log(\rho)\|_\rho^2 \;.
    \end{equ}
    Completing the square in \eqref{eq:energy_log_mean_duality} gives the variational representation
    \begin{equs}[eq:fisher_dual_nc]
        \CI_\omega(\rho) &= \sup_{\substack{a\in D(\overline{\nabla})\\ a=a^*}} \left\{2\operatorname{Re}\CE(\rho,a) - \|\overline{\nabla}a\|_\rho^2 \right\} \;.
    \end{equs}
    Indeed, the expression inside the supremum equals
    \begin{equ}
        \CI_\omega(\rho) - \|\overline{\nabla}a - \overline{\nabla}\log(\rho)\|_\rho^2 \;,
    \end{equ}
    and equality is attained at $a=\log(\rho)$.

    We next deduce Fisher-information decay. Since $P_t$ is self-adjoint and commutes with the functional calculus of $-\Lin_V$,
    \begin{equs}[eq:form_semigroup_duality]
        \CE(P_t\rho,a) &= \Braket{(-\Lin_V)^{1/2}P_t\rho, (-\Lin_V)^{1/2}a}_{\CL^2(\omega)} =\\
    &= \Braket{(-\Lin_V)^{1/2}\rho, (-\Lin_V)^{1/2}P_ta}_{\CL^2(\omega)} = \CE(\rho,P_ta) \;.
    \end{equs}
    Positivity and unitality give $\eps\bone\leqslant P_t\rho\leqslant R\bone$. Applying \eqref{eq:fisher_dual_nc}, \eqref{eq:form_semigroup_duality} and \eqref{eq:lmge}, we obtain
    \begin{equs}[eq:fisher_decay_nc]
        \CI_\omega(P_t\rho) &= \sup_{\substack{a\in D(\overline{\nabla})\\ a=a^*}} \left\{2\operatorname{Re}\CE(P_t\rho,a) - \|\overline{\nabla}a\|_{P_t\rho}^2 \right\} \leqslant\\
    &\leqslant \sup_{\substack{a\in D(\overline{\nabla})\\ a=a^*}} \left\{2\operatorname{Re}\CE(\rho,P_ta) - e^{2Kt}\|\overline{\nabla}P_ta\|_\rho^2 \right\} =\\
    &= e^{-2Kt} \sup_{\substack{a\in D(\overline{\nabla})\\ a=a^*}} \left\{2\operatorname{Re} \CE\left(\rho,e^{2Kt}P_ta\right) - \left\| \overline{\nabla}\left(e^{2Kt}P_ta\right) \right\|_\rho^2 \right\} \leqslant\\
    &\leqslant e^{-2Kt}\CI_\omega(\rho) \;.
    \end{equs}

    We now identify the entropy at the endpoint of the flow. Since $P_t\bone=\bone$ and $\Theta_\bone=\bone^{\op}\otimes\bone$, applying \eqref{eq:lmge} with weight $\bone$ gives
    \begin{equ}
        \|\overline{\nabla}P_t\rho\|_{\CL^2(\omega^e)^m}^2 \leqslant e^{-2Kt} \|\overline{\nabla}\rho\|_{\CL^2(\omega^e)^m}^2 \;.
    \end{equ}
    Since $\omega(P_t\rho)=1$, \eqref{eq:closed_PI_for_MLSI} implies
    \begin{equ}[eq:L2_density_decay]
        \|P_t\rho-\bone\|_{\CL^2(\omega)}^2 \leqslant C_{\mathrm P}e^{-2Kt} \|\overline{\nabla}\rho\|_{\CL^2(\omega^e)^m}^2 \xrightarrow{t\to\infty} 0 \;.
    \end{equ}
    The scalar inequality $r\log r-r+1 \leqslant (r-1)^2$ for $r\geqslant0$ and functional calculus give
    \begin{equ}
        0 \leqslant \Ent_\omega(P_t\rho) \leqslant \|P_t\rho-\bone\|_{\CL^2(\omega)}^2 \xrightarrow{t\to\infty} 0 \;.
    \end{equ}

    Set $\rho_t\eqdef P_t\rho$. The power-series expansion of the logarithm shows that $t\mapsto\Ent_\omega(\rho_t)$ is locally absolutely continuous and, for almost every $t>0$,
    \begin{equ}[eq:entropy_dissipation_nc]
        \frac{\d}{\d t}\Ent_\omega(\rho_t) = \omega\left((\bone+\log(\rho_t))\Lin_V\rho_t\right) = -\CE(\log(\rho_t),\rho_t) = -\CI_\omega(\rho_t) \;.
    \end{equ}
    Here the final equality uses the symmetry of $\CE$ and $\CI_\omega(\rho_t)=\CE(\rho_t,\log(\rho_t))$. For $0<\delta<T$, integration of \eqref{eq:entropy_dissipation_nc} gives
    \begin{equ}
        \Ent_\omega(P_\delta\rho)-\Ent_\omega(P_T\rho)=\int_\delta^T\CI_\omega(P_t\rho)\,\d t \;.
    \end{equ}
    Strong continuity gives $P_\delta\rho\to\rho$ in $\CL^2(\omega)$ as $\delta\downarrow0$, while positivity and unitality give $\eps\bone\leqslant P_\delta\rho\leqslant R\bone$. Polynomial approximation of $r\mapsto r\log r$ on $[\eps,R]$ therefore yields $\Ent_\omega(P_\delta\rho)\to\Ent_\omega(\rho)$. Since $\CI_\omega(P_t\rho)\geqslant0$, monotone convergence now gives, for every $T>0$,
    \begin{equ}
        \Ent_\omega(\rho)-\Ent_\omega(P_T\rho) = \int_0^T \CI_\omega(P_t\rho) \,\d t \leqslant \CI_\omega(\rho) \int_0^T e^{-2Kt}\,\d t = \frac{1-e^{-2KT}}{2K} \CI_\omega(\rho) \;.
    \end{equ}
    Letting $T\to\infty$ proves \eqref{eq:MLSI_nc} for regular densities.

    Finally, let $\rho$ be an arbitrary density with $\CI_\omega(\rho)<\infty$. By \eqref{eq:relaxed_fisher}, there is a sequence $R_j\to\infty$ such that
    \begin{equ}
        \CI_\omega(\rho\wedge R_j)\xrightarrow{j\to\infty}\CI_\omega(\rho) \;.
    \end{equ}
    In particular, $\rho\wedge R_j\in D(\overline{\nabla})$. Set
    \begin{equ}
        c_j \eqdef \omega(\rho\wedge R_j) \;, \qquad \rho_j \eqdef \frac{\rho\wedge R_j}{c_j} \;, \qquad \rho_{j,\eps} \eqdef \frac{\rho_j+\eps\bone}{1+\eps} \;.
    \end{equ}
    Each $\rho_{j,\eps}$ is a bounded uniformly positive density, and in particular $\rho_{j,\eps},\log(\rho_{j,\eps}) \in D(\overline{\nabla})$.
    Homogeneity of the form and the definition \eqref{eq:relaxed_fisher_bounded} yield
    \begin{equ}
        \CI_\omega(\rho_{j,\eps}) = \frac{1}{1+\eps} \CE\left(\rho_j, \log(\rho_j+\eps\bone)\right) \xrightarrow{\eps\downarrow0} \CI_\omega(\rho_j) \;.
    \end{equ}
    Moreover, bounded functional calculus gives
    \begin{equ}
        \Ent_\omega(\rho_{j,\eps}) \xrightarrow{\eps\downarrow0} \Ent_\omega(\rho_j) \;.
    \end{equ}
    Applying the regular-density inequality and then sending $\eps\downarrow0$ gives
    \begin{equ}
        \Ent_\omega(\rho_j) \leqslant \frac{1}{2K}\CI_\omega(\rho_j) \;.
    \end{equ}
    Since $c_j\to1$, \eqref{eq:fisher_homogeneity} and the choice of $(R_j)_j$ imply
    \begin{equ}
        \CI_\omega(\rho_j)=\frac{\CI_\omega(\rho\wedge R_j)}{c_j}\xrightarrow{j\to\infty}\CI_\omega(\rho) \;.
    \end{equ}
    Finally,
    \begin{equ}
        \Ent_\omega(\rho_j) = \frac{\omega\left((\rho\wedge R_j)\log(\rho\wedge R_j)\right)}{c_j} - \log c_j \xrightarrow{j\to\infty} \Ent_\omega(\rho) \;.
    \end{equ}
    Passing to the limit proves \eqref{eq:MLSI_nc} for every density. If $\CI_\omega(\rho)=\infty$, the inequality is automatic.
\end{proof}

\begin{corollary}[Homogeneous Form of the Modified Logarithmic Sobolev Inequality]
    \label{cor:homogeneous_mlsi} For every $f\in\CL^1(\omega)_+$,
\begin{equ}
    \omega(f\log f) - \omega(f)\log\omega(f) \leqslant \frac{1}{2K}\CI_\omega(f) \;.
\end{equ}
\end{corollary}

\begin{proof}
    
    If $\omega(f)>0$, apply Theorem~\ref{thm:lmge_implies_mlsi} to $\rho=f/\omega(f)$ and use \eqref{eq:fisher_homogeneity} together with the homogeneity of entropy:

\begin{equ}
        \Ent_\omega(f) = \omega(f)\Ent_\omega(\rho) \;, \qquad \CI_\omega(f) = \omega(f)\CI_\omega(\rho) \;. 
\end{equ}

    If $\omega(f)=0$, faithfulness implies $f=0$. 
\end{proof}

\begin{remark}[Relation to Gross' Inequality] 
    Equation~\eqref{eq:MLSI_nc} is the density-form modified logarithmic Sobolev inequality, whose dissipation is $\CE(\rho,\log(\rho))$. It is not the Gross logarithmic Sobolev inequality for $a^*a$ and therefore does not, by itself, imply hypercontractivity. A hypercontractivity conclusion requires a separate Gross inequality or an additional comparison between the two dissipation functionals. 
\end{remark} 

\bibliographystyle{amsplain}
\bibliography{NonCommutativeEntropy}

\end{document}